\documentclass{article}
\usepackage{amssymb,epsfig,newlfont,amsmath}
\usepackage{color}
\usepackage{amssymb,epsfig,newlfont,amsmath}
\usepackage[french,english]{babel}
\usepackage[T1]{fontenc}
\usepackage{amssymb}
\usepackage{amsmath}
\usepackage{amscd}

\newtheorem{theorem}{Theorem}[subsection]
\newtheorem{lemma}[theorem]{Lemma}
\newtheorem{proposition}[theorem]{Proposition}

\catcode`\Ž=13
\catcode`\=13
\catcode`\ˆ=13
\catcode`\=13
\catcode`\=13
\catcode`\‰=13
\catcode`\™=13
\catcode`\"=13
\catcode`\•=13
\catcode`\ž=13
\catcode`\=13
\catcode`\'=13
\catcode`\Ï=13
\catcode`\Ë=13

\def Ž{\'e}
\def {\`e}
\def ˆ{\`a}
\def {\`u}
\def {\^e}
\def ‰{\^a}
\def ™{\^o}
\def "{\^{\i}}
\def •{\"{\i}}
\def ž{\^u}
\def {\c c}
\def '{\"e}
\def Ï{\oe}
\def Ë{\`A}

\def\math#1{{\mathbb{#1}}}

\def\RM{\math R}
\def\QM{\math Q}
\def\CM{\math C}
\def\ZM{\math Z}

\def\bs{\backslash}

\def\Hom{\,\mathrm{ Hom}} 
\def\vol{\,\mathrm{ vol}\,\,}

\def\Im{\,\,\mathrm{ Im}}

\def\det{\,\mathrm{ det}}

\def\Re{\mathrm{ Re}}

\def\ade{{\math A}}

\def\mun{^{-1}}

\def\vf{\varphi}
\def\vp{\varpi}

\def\G{G}
\def\P{P}
\def\H{{H}}
\def\N{{N}}
\def\GG{{\mathbf G}}
\def\PP{{\mathbf P}}
\def\QQ{{\mathbf Q}}

\def\NN{{\mathbf N}}

\def\AG{{\mathfrak A}_G}
\def\AP{{\mathfrak A}_P}
\def\APO{{\mathfrak A_{P_0}}}

\def\goth{\mathfrak} 
\def\ga{\goth a}
\def\ag{\goth a_G}
\def\ap{\goth a_P}
\def\apo{\goth a_{P_0}}
\def\aq{\goth a_Q}
\def\ar{\goth a_R}
\def\iapg{\Lambda_P}
\def\iapgo{\Lambda_{P_0}}

\def\I{\mathcal I}
\def\B{\mathcal B}
\def\pni{\par\noindent}
\def\CMC{{\mathcal C}}
\def\FMC{{\mathcal F}}
\def\DMC{{\mathcal D}}
\def\GMC{{\mathcal G}}
\def\AMC{\mathcal A}
\def\AMCP{{\mathcal A'}}
\def\PMC{\mathcal P}
\def\QMC{\mathcal Q}

\def\HMC{\mathcal H}
\def\HMCP{{\mathcal H_P}}
\def\LMC{\mathcal L}
\def\LMCP{{\LMC_P}}

\def\HMCPA{{\HMC_P^0}}
\def\LMCPO{{\LMC_{P_0}}}
\def\LMCQ{{\LMC_Q}}
\def\LMCD{\LMC_{\mathrm{disc}}}
\def\LMCC{\LMC_{\mathrm{cusp}}}
\def\Mint{{M}}
\def\mint{{m}}

\def\Ei{E}
\def\domfond{\mathfrak D}

\def\Mlev{\mathbf M}
\def\hvf{\hat\vf}
\def\trunc{\mathbf{\Lambda}^T}
\def\weyl{W}
\def\Zeta{L}
\def\ptf{\,\,.}
\def\com#1{\!\!\quad\hbox{#1}\quad\!\!}
\def\comm#1{\qquad\hbox{#1}\qquad}

 \title{The Langlands spectral decomposition}
\author{ J.-P. Labesse\\Institut Math\'ematique de Luminy\\ UMR 7373}
\date{}

\begin{document}
\maketitle

\begin{abstract}

We  review the standard definitions for basic objects in automorphic theory
 and then give an overview of Langlands fundamental results established  in \cite{LES}. We try to explain
ideas behind the proof when reasonably simple following mainly the surveys \cite{LB2} and \cite{A1}.
We emphasize the role of the truncation operator which appears for the first time, but in some guise, in \cite{LB2}.
In the last sections we explain  the formal aspects of the spectral decomposition for the space of
 $K$-invariant functions on $GL(2)$ and $GL(3)$ being otherwise rather sloppy on analytic questions.

We assume the reader familiar with basic representation theory, linear algebraic groups and adles.
We must apologize for copying, most of the time, parts of \cite{LB2} and \cite{A1}.
We only have given in greater details some elementary arguments
that were maybe a bit too sketchy in these references. On the other hand we made no attempt to be more
explicit than these surveys on the most difficult part of the proof.

\end{abstract}

\tableofcontents

\section{Introduction}

\subsection{Langlands fundamental paper}
Consider a connected reductive group $\GG$ defined over $\QM$
and $\Gamma$ an arithmetic subgroup of the connected reductive Lie group
$\G=\GG(\RM)^0$.
The spectral decomposition of the right regular representation of $G$ on
the Hilbert space $\LMC$ of square integrable function on the quotient space  $\Gamma\bs \G$:
$$\LMC=L^2(\Gamma\bs \G)$$
 is one of the two most important results of the long paper  written by Langlands
in the early sixties and completed in 1964
under the title: 
\par\hfil{\bf On the Functional Equations Satisfied by Eisenstein Series}. \pni
The second main result is given in the title. Both rely on the analytical continuation
of Eisenstein series $E(x;\Phi,\lambda)$ and of intertwining operators 
$\Mint(s,\lambda)$ that show up in their functional equations. 
A mimeographed version  
of this long and difficult manuscript was circulated 
but remained unpublished until 1976 when it appeared, together with a Preface
and four appendices added, as Springer Lecture Notes 544 \cite{LES}.  

This almost  explicit description  (see \ref{rp} below) of the spectral decomposition obtained by Langlands
is a basic tool for the study of automorphic forms, in particular to establish the  
Trace Formula which, in turn, is one of the most powerful tools toward this study. In fact,
already in the early sixties, Langlands had in mind to work on a generalization of Selberg's
Trace Formula to all reductive group.
Moreover, instances of  intertwining operators $\Mint(s,\lambda)$
were soon recognized as expressible in terms of a new family of $L$-functions \cite{LEP}
and this, in turn, was at the origin of the Functoriality Conjectures \cite{LLMA}.

\subsection{Classical versus adŽlic framework}

The original monograph  
is written in the classical (non adŽlic) language and deals
with discrete subgroups $\Gamma$ of $\G$
 satisfying  technical assumptions that are automatically satisfied when 
 $\Gamma$ is an arithmetic subgroup and {\it a fortiori} when $\Gamma$ is a congruence subgroup.
General arithmetic subgroups are important in geometry while congruence subgroups and Hecke operators 
play a central role in number theory.
The consideration of Hecke operators leads naturally to work with 
projective limits of coverings defined by congruence subgroups and
 this essentially amounts to work with $\GG(\QM)\bs \GG(\ade)$
 but thus we by-pass more general arithmetic subgroups. 
 It turns out to be simpler to deal with adŽlic quotients. 
 Moreover it is also necessary to use adŽlic language for the formulation of
 Langlands functoriality principles. We observe that the
appendices II, III and IV added in \cite{LES}
by Langlands in the seventies to the original monograph  
are written in adŽlic language. From now on, we restrict ourselves to the adŽlic setting,
nevertheless we shall use a set of notation similar to what is used in the classical setting.

\subsection{Remarks on the strategy}

An abstract non sense tells us that  the spectral decomposition of $\LMC$
can be described in term of 
generalized vectors $E(\bullet;\pi)$
in the sense of Gelfand: any $\vf\in\LMC$ can be written
at least in a formal way
$$\vf(x)=\int_{\pi\in\Pi} \hvf(\pi) E(x;\pi)\,d\mu(\pi)\comm
{with}\hvf(\pi)=<\vf,E(\bullet;\pi)>_\LMC$$
where $\Pi$ parametrizes a set data attached to unitary representations of $G$.
Langlands shows that these generalized vectors can be constructed via the meromorphic
continuation of the so-called Eisenstein series.
When $\GG=GL(1)$ over $\QM$,
Eisenstein series are elementary objects, namely they are characters of the form 
$$E(x;\Phi,\lambda)=|x|^\lambda\Phi(x)$$
with $\lambda\in i\RM$ and $\Phi$ a Dirichlet character. 
In such a case the spectral decomposition is nothing but an instance of
Pontryagin theory (i.e. Fourier inversion for locally compact abelian groups).
In general Eisenstein series are of the form (see below for notation)
$$E(x;\Phi,\lambda) =\sum_{\gamma\in\Gamma_P\bs\Gamma}e^{<\lambda+\rho_P,H_P(\gamma x)>}
\Phi(\gamma x)$$ where 
 $\lambda$ is a parameter in the complexification of a finite dimensional real vector space
 and $\Phi$ is an automorphic form on a parabolic subgroup. 
  These series converge in some domain but
may not converge for purely imaginary values 
of $\lambda$ we are interested in, then one needs to establish
their meromorphic continuation.

For Eisenstein series on groups of $\QM$-rank one, a proof of the meromorphic continuation 
has been obtained by Selberg (see \cite{S1} for a brief account). There are various independent approaches.
New insights were necessary to deal with the general case and
Langlands proof is quite involved. He proves rather directly, 
in chapter 6 of \cite{LES}, the meromorphic continuation and the functional equations
satisfied by intertwining operators and
Eisenstein series  when $\Phi$ is cuspidal by a reduction to the rank one case.
But in general, when $\Phi$ is only assumed to belong to the discrete spectrum of a parabolic subgroup, 
the meromorphic continuation is obtained at the same time as the
full spectral decomposition in the very difficult chapter 7.  In fact, this chapter is famous for being
{\it almost impenetrable}, as said by Langlands himself in his Preface to the Springer Lecture Notes.

\subsection{Further references}

A brief survey of \cite{LES}
is to be found  in Langlands article \cite{LB2} for the AMS Conference in Boulder in 1965.   
Another survey, due to J. Arthur \cite{A1}, was written  for the AMS Conference in Corvallis  in 1977.
 The easiest part of the proof i.e. the results of the first six
chapters of \cite{LES} may also be found in Harish-Chandra's Lecture Notes \cite{H}.
The ``Paraphrase''  by MÏglin and Waldspurger  \cite{MW2} 
gives a complete and detailed proof
in the adŽlic language for groups over arbitrary global fields
(i.e. number fields or function fields), also valid for metaplectic groups.

\subsection{Recent progress}\label{rp}

The spectral decomposition is explicitly given by Langlands in terms of two black boxes.
The first one is the cuspidal spectrum of Levi subgroups. 
The only information given is that the cuspidal spectrum is discrete with finite multiplicity.
Little progress have been made and the conjectures describing 
this spectrum in term of a dual object are still mainly out of reach.
The second black box is the residual spectrum arising from poles of Eisenstein series i.e.
the discrete but non cuspidal spectrum. Already for $GL(n)$ with
$n\ge 5$ the combinatorics of the residues is so involved that the explicit description
of the residue spectrum, conjectured by Jacquet,
 was only achieved by MÏglin and Waldspurger in \cite{MW1}.
For classical groups recent progress by MÏglin, Arthur and others
have been made, but for arbitrary reductive groups the goal is still out of reach.
A result, due to Franke \cite{Fr}, shows that all automorphic forms are finite linear combination
of Eisenstein series up to maybe taking residues and derivatives.
A different proof for the meromorphic continuation of Eisenstein series has been announced by Bernstein
a long time ago and is the subject of a recent paper \cite{BL}.

\subsection*{Acknowledgements}  I am grateful to Dinakar Ramakrishnan and to Jean-Loup Waldspurger for 
useful suggestions and criticisms.

\section{Basic objects and main theorems}

\subsection{Notation} 
 If $\PP$ is a connected linear algebraic group over $\QM$
we write $P$ for $\PP(\ade)$  where $\ade$ is the ring of adles of $\QM$,
$P_\infty$ for the real Lie group $\PP(\RM)$, $P_f$ for the group $\PP(\ade_f)$ of points over the finite adles
$\ade_f$
and $\Gamma_P$ for the group of rational points $\PP(\QM)$.
Let $\mathfrak X(P)$ be the group of rational characters of $P$ and consider
$$\ap=\Hom(\goth X(P),\RM)\ptf$$ 
This is a real vector space whose dimension is denoted $a_P$.
We denote by 
$$\H_P:P\to\ap$$  the  map
induced by $$x\mapsto \left(\chi\in\goth X(P)\mapsto\log(|\chi(x)|)\right)$$
and by $P^1$ its kernel.   Since we are dealing with a number field the map $H_P$ is surjective
(this would not be true for function fields) and there exists a section  of the surjective map $H_P$
 with values in $\AP$, the connected component of the group of real points of the
 maximal $\QM$-split torus $A_P$ in the center of a Levi subgroup $M_{P}$ in $P$:
 $$\AP=A_P(\RM)^0\subset M_P\ptf$$
  In other words $H_P$ induces an isomorphism
 $$\AP\to\ap\ptf$$
The evaluation of a linear form $\lambda\in\ap^*\otimes\CM$ on a vector $H\in\ap$
will be denoted $<\lambda,H>$.

\subsection{Parabolic subgroups and Iwasawa decomposition}

Now consider  a connected reductive group $\GG$ 
defined over $\QM$. According to our convention
we put $\G=\GG(\ade)$ and $\Gamma_G=\GG(\QM)$ or even simply $\Gamma$ if no confusion may arise. 
As in the classical setting $\Gamma$
is a discrete subgroup in $\G$ and the quotient $\Gamma\bs\G$ is of finite volume when
$\GG$ is semisimple.

We choose a minimal rational parabolic subgroup $P_0$ in $G$
and a Levi decomposition $$P_0=M_0N_0\ptf$$
Consider a standard (rational) parabolic subgroup $P$  with Levi decomposition 
$P=M_P N_P$ where  $M_0\subset M_P$.  
We denote by $\Delta_P$ the set of non zero elements in the projection on
$(\ap/\ag)^*$ of the set $\Delta_{P_0}$ of
simple roots.  We denote by $\rho$ the half sum of positive roots, by
$\rho_P$ its projection on $(\ap/\ag)^*$. We shall, from time to time, use
the bilinear scalar product on $(\ap/\ag)^*\otimes\CM$ deduced from the Killing form
and we shall denote by $<\lambda,\mu>$ its value on the couple $(\lambda,\mu)$.
This notation implicitly implies that we identify $\ap/\ag$ with its dual via this bilinear form.

We choose  once for all a ``good''
maximal compact subgroup $K$ of $\G$ i.e. of the form
$K=K_\infty\times\prod_p K_p$ where $K_\infty$ is a maximal compact subgroup in $G_\infty$,
 $K_p$ is a special in $G_p=\GG(\QM_p)$
for  all prime $p$ and hyperspecial for almost all  $p$ (see \cite{Tits} for these notions). For example;
when $\GG=GL(n)$ we take $K_p=GL(n,\ZM_p)$.
Then we have 
the Iwasawa decomposition
$$\G=PK=N_P M_P K\ptf$$ 
The homorphism $H_P:P\to\ap$ is trivial on $N_P$ and  on $K\cap P$. It
extends to a function from $\G$ onto $\ap$,
again denoted $H_P$, satisfying
 $$ H_P(pk)=H_P(p)\quad \hbox{for}\quad p\in P\quad \hbox{and}\quad k\in K\ptf$$
 This allows to view functions on $\ap$ as functions on $\G$.
The Weyl chamber in $\mathfrak a_{P_0}/\ag$ is the cone
defined by the inequalities 
$$<\alpha,H> >0\comm{for}\alpha\in \Delta_{P_0}\ptf$$
Fundamental weights
define another cone whose characteristic function is denoted $\hat\tau_{P_0}$. More generally
one defines in a similar way characteristic functions   $\hat\tau_P$ of cones in $\ap/\ag$
for any parabolic subgroup using   $\check\Delta_P$ the  basis dual to the basis of co-roots
(see \cite{A1} or \cite{LW}).  They will appear below in the definition of the truncation operator.

 \subsection{The right regular representation}
 
In the following we shall only consider functions invariant under $\AG$. Endowed with a $G$-invariant measure, the quotient
$$X_G= \AG\Gamma\bs \G\simeq \Gamma\bs \G^1$$ is of finite volume \cite{B1}.
Let $\LMC$ be the Hilbert space of square integrable functions on $X_G$:
$$\LMC=L^2(X_G)\ptf$$ 
We want to understand the spectral decomposition $\LMC$ under
of the right regular representation 
$R$ of $\G$ in $\LMC$ i.e. the map  $\G\times\LMC\to\LMC$ defined by
$$(x,\vf)\mapsto R(x)\vf\quad \hbox{where}\quad(R(x)\vf)(y)=\vf(yx)\quad \hbox{for}\quad x\in\G,
\quad y\in X_G
\quad \hbox{and}\quad \vf\in\LMC$$
where, by abuse of notation, we use the same letter for an element of $G$ and its image in $X_G$.
Consider now a smooth compactly supported
function $f$ on $G$. An operator $R(f)$ is
defined by integration of $f$ against the right regular representation:
$$(R(f)\vf)(y)=\int_{G} f(x)\vf(yx)\,dx\ptf$$
Here $dx$ is some Haar measure on $G$.  The spectral decomposition of such operators, which is
intimately related to the spectral decomposition of $\LMC$,
is a tool and a goal of automorphic theory.
The main concern for the spectral decomposition is to understand the discrete
spectrum $$\LMCD:=L^2_{\mathrm{disc}}(X_G)$$
which is the Hilbert direct sum of irreducible subspaces in $\LMC$.
The trivial representation $\mathbf 1_G$ is an obvious but already quite
interesting constituent. The way it appears in the spectral decomposition
via residues of Eisenstein series leads to the proof of Weil's conjecture
for Tamagawa numbers (see \ref{triv}).

\subsection{Constant terms and truncation operators} Two operations
 play an essential role in the theory of Eisenstein series and of trace formula as well.
The simplest one is the computation of the constant term
along a parabolic subgroup $P$: $\vf\mapsto\vf_P$. The second one is the truncation 
operator $\vf\mapsto\trunc\vf$.
 The definition of the truncation operator in general is due to J. Arthur \cite{A2}.
Its properties, recalled below, rely on reduction theory and combinatorial arguments already present in \cite{LB2}.
In fact, generalizing an operation used in Selberg's approach for $\QM$-rank one groups
in \cite{S1}, Langlands  constructs in sections  8 and 9 of  \cite{LB2},
variants of Eisenstein series attached to cusp forms,
 denoted $E''(\bullet,\Phi,\lambda)$, that are nothing but  $\trunc E(\bullet,\Phi,\lambda)$
and he computes the scalar product of two such functions. He thus gets a formula,
recalled in \ref{TrSc},
 which is  an explicit and rather simple expression in term of intertwining operators. This
 is a key tool for the analytic continuation of Eisenstein series:
it plays a role similar to the  Maass-Selberg relations used by Harish-Chandra's in Chapter IV of \cite{H}.
Although not fully recognized, this implicitly uses the full power of  truncation operators
and  $(G,M)$-families\footnote{This is a combinatorial technique implicit in Langlands' papers 
and used extensively by Arthur in his work on the Trace Formula. See section 1.10 in \cite{LW}
for a synthetic account.}, 
two techniques that play an essential role in establishing the Trace Formula,
a fact of which Langlands was already aware in the early sixties: 
he says it explicitly in section 7 p.~243 of \cite{LB2}.

These operations are elementary when dealing with modular forms. Consider
a modular form $f$ of even weight $k$
for $GL(2,\ZM)$ whose $q$-expansion in the upper half plane $\mathfrak H$
is of the form
$$f(z)=\sum_0^\infty a_nq^n\comm{with}q=e^{2\pi iz}\ptf$$ 
Its constant term is  $a_0$ and its truncated avatar is the function 
 in the fundamental domain:
 $$\domfond =\{z\in\CM\,\,|\,\,\,|\Re(z)|\le1/2\com{and} |z|\ge 1\}$$
 equal to
$$f(z)\com{if $\Im(z)\le e^T$}\!\!\!\!\!\!\!\!\com{and} f(z)-a_0
\com{if} \Im(z)\!> e^T\com{for}z\in\mathfrak D
$$ 
when $T$ is a positive real number. The modular form $f$ corresponds to an
automorphic form $\vf$ on $$Y=
GL(2,\QM)\bs GL(2,\ade)/GL(2,\overline\ZM)\simeq
GL(2,\ZM)\bs GL(2,\RM)
$$ defined by
$$\vf(g)=J(g,i)^kf(g.i)
\com{for}g=  \begin{pmatrix}a&b\cr c&d\cr\end{pmatrix}
\in GL(2,\RM)
$$
with
$$
g.i=\frac{ai+b}{ci+d}=z\in \CM-\RM
\com{and} J(g,\tau)^k=(\!\det(g))^{k/2}(c\tau+d)^{-k}
\ptf$$
The constant term $\vf_P$ is given by $J(g,i)^ka_0$ and $\Lambda^T\vf$ is the function 
on $Y$ such that
$$\Lambda^T\vf(g)=\vf(g)\com{if} \!\!\!\Im(g.i)\le e^T \!\!\!\com{and} 
\Lambda^T\vf(g)=(f(g.i)-a_0)J(g,i)^k\com{if} \!\!\!\Im(g.i)> e^T$$  
 when $T$ is a positive real number and  $g.i=z\in\domfond$.

Now return to the general case.  Taking the constant term 
along a parabolic subgroup $P$ is the operation which transforms
a locally integrable function $\vf$  on $X_G$ into a function  $\vf_P$  on
$$X_{P,G}=\AG\Gamma_P N_P\bs G$$
defined by the integral
$$\vf_P(x)=\int_{\Gamma_P\cap{N_P}\bs\N_P}\vf(nx)\,dn$$
where $dn$ gives measure 1 to $\Gamma_P\cap{N_P}\bs\N_P$.
The following formal computation plays a fundamental role: 
\begin{lemma}\label{fc} 
Let  $\vf$ be a function on $X_G$ and
$\phi$ a function on $X_{P,G}$.  Denote by $E_\phi$
the function on $X_G$ defined by
$$E_\phi(x)=\sum_{\Gamma_P\bs\Gamma_G}\phi(\gamma x)$$
provided the sum is convergent.
Then, if integrals are convergent,
$$<E_\phi,\vf>_{X_\G}=\int_{X_G}E_\phi(x)\overline
{\vf(x)}\,dx=\int_{\AG\Gamma_P\bs \G}\phi(x)\overline
{\vf(x)}\,dx =\int_{X_{P,G}}\phi(x)\overline
{\vf_P(x)}\,dx \ptf$$
\end{lemma}

The truncated function $\trunc\vf$ is an alternate sum 
 indexed by standard parabolic subgroups of series over $\Gamma_P\bs\Gamma_G$
 of terms that are products of characteristic functions $\hat\tau_P$
of  cones in $\ap$ translated by some parameter $T\in\mathfrak a_{P_0}$ 
(or rather its projection on $\ap/\ag$) times constant terms $\vf_P$:
$$\trunc\vf(x)=\sum_{P\supset P_0}(-1)^{a_P-a_G}
\sum_{\gamma\in\Gamma_P\bs\Gamma_G}
\hat\tau_P(H_P(\gamma x)-T)\vf_P(\gamma x)\ptf$$
The series are trivially convergent since reduction theory shows that 
given a compact set $\Omega$, then for  $x\in\Omega$
there is only a finite number of $\gamma\in\Gamma_P\bs\Gamma_G$ such that the expression
$\hat\tau_P(H_P(\gamma x)-T)$ does not vanish and
moreover it will vanish identically if $P$ is a proper subgroup and if $T$ is far enough 
from the walls of the Weyl chamber.

\begin{proposition}\label{ident}
\pni (i)
 Given a compact set $\Omega\subset\G$ then,  provided $T$ is far enough 
from the walls of the Weyl chamber,
 one has for any locally integrable function  $\vf$:
$$\trunc\vf(x)=\vf(x)\comm{for all}x\in\Omega \ptf$$
\pni(ii)  $\trunc$ induces a self adjoint idempotent operator on the Hilbert space $\LMC$
(i.e.  an orthogonal projector) 
$$\trunc=(\trunc)^*=(\trunc)^2$$
\pni(iii)  $\trunc$ transforms functions of uniform moderate growth into rapidly decreasing functions.
In particular, given a smooth compactly supported function
$f$ on $\G$  the compositum $$\trunc\circ R(f)$$ is an operator of Hilbert-Schmidt type.
\end{proposition}

\begin{proof}
Assertion (i) follows immediately from the above remarks.  For a proof of assertion (ii) and (iii) we refer to 
refer  to Chapters 4 and 5 of \cite{LW}\footnote{%
The reader should be warned that the proofs given in \cite{A2} have to be slightly
corrected when dealing with arbitrary reductive groups since some arguments
may not apply for non split groups.}.
\end{proof}

\subsection{Automorphic forms}

Automorphic forms are functions $\vf$ on $X_G$ that are smooth, $K$-finite, of uniform moderate growth
and annihilated by an ideal $\mathfrak I$ of finite codimension in the center $\goth z(\goth g_\infty)$
of the enveloping algebra of the Lie algebra $\goth g_\infty$ of $G_\infty$:
$$\vf\star\goth c=0\comm{for all}\goth c\in\mathfrak I\ptf$$
Here $\star$ denotes the convolution product and $\goth c$ is viewed as a distribution supported
 at the origin in $G$.
  The $K$-finiteness may be expressed by asking
that $$\vf\star e=\vf$$ where $e$ is an idempotent defined by a measure supported
on $K$ associated to a finite dimensional representations of $K$.
For a more detailed definition we refer to  \cite{B2}, \cite{BJ}, \cite{H} or \cite{MW2}.
An important property is that the space of automorphic forms $\vf$ 
annihilated by any given ideal $\mathfrak I$
and such that $\vf\star e=\vf$  for any given $e$, is a finite dimensional vector space.
This implies (see \cite{H} for example) that given an automorphic form $\vf$ there exists
a smooth $K$-finite function $f$ compactly supported on $X_G$ such that 
$$\vf\star f=\vf\ptf$$
Since automorphic forms are of uniform moderate growth
the truncation operator transforms any automorphic form $\vf$ into a rapidly decreasing 
function $\trunc\vf$ and in particular one has $\trunc\vf\in\LMC$.

\subsection{Cusp forms}

An automorphic form $\vf$ is said to be cuspidal
(or a cusp form) if for any  parabolic subgroup $P\ne\G$ 
the constant terms $\vf_P$ vanish identically. 
It suffices to check this for standard parabolic subgroups.
A key observation is that
the truncation operator acts by the identity on a cusp form $\vf$:
$$\trunc\vf=\vf\ptf$$
 Then, assertion (iii) of \ref{ident} implies that cusp forms are rapidly decreasing on $X_G$ 
 and in particular are square-integrable.
They generate in $\LMC$ the cuspidal spectrum $$\LMCC:=L^2_{\mathrm{cusp}}(X_G)\ptf$$ 

Consider a smooth compactly supported
function $f$ on $G$. The  restriction of $R(f)$  to $\LMCC$ 
is a trace class operator; this follows from three observations: 
\pni (i)  the factorization theorem of Dixmier-Malliavin
 tells us that $f$ can be written as a finite sum of convolution products:
$$f=f_1\star f_2+\cdots +f_{2r-1}\star f_{2r}$$
for some integer $r$,
\pni (ii)  $\trunc R(f_i)\,R(f_j)\trunc$ being a product of two Hilbert-Schmidt operators is of trace class,
\pni (iii)  $\trunc R(f)\trunc\vf=R(f)\vf$\,\,\,
when $\vf$ is cuspidal.

This implies that the cuspidal spectrum
$\LMCC$ decomposes as a discrete sum with finite multiplicities 
of irreducible representation and hence is a subspace of the discrete spectrum
$\LMCD$.
This result was first obtained by Gelfand and Piatetskii-Shapiro.
Unless $X_G$ is compact,  $\LMCC$ 
is only a strict subspace of $\LMCD$: in fact, whenever $X_G$ is not
compact, the trivial representation is discrete but not cuspidal.

\subsection{Automorphic forms  on parabolic subgroups}\label{cuspapra}

Consider a standard parabolic subgroup $P$ and the quotient space 
$$X_P=\AP\Gamma_P N_P\bs G\ptf$$
 Let us denote by 
 $$\LMCP=L^2(X_P)
 $$
 the Hilbert space generated by functions that are square 
integrable on $X_P$ for the right $G$-invariant measure:
$$<\Phi,\Phi>_\LMCP=\int_K\int_{\Gamma_P N_P\bs P^1} 
\Phi(pk)\overline{\Phi(pk)}\,dp\,dk\ptf$$
The Haar measure $dk$ is normalized so that $\vol(K)=1$ and $dp$ is a
 right $P^1$-invariant measure.
The space $X_P$  is of finite volume. Observe that $\LMC_G=\LMC$.
 
 We say that a function $\Phi$ on  $X_{P}$ is automorphic on $\P$
if it is $K$-finite and if, for all $x\in\G$, the functions 
$m\mapsto\Phi(mx)$ for $m\in M_P$ are automorphic forms on the Levi subgroup of $P$.
We say that an automorphic form $\Phi$  is cuspidal on  $\P$  if, moreover, the functions 
$m\mapsto\Phi(mx)$  are cusp forms.

\subsection{Representations $I_{P,\lambda}$ and intertwining operators}

Given $\Phi$ on $X_P$ and $\lambda\in(\ap/\ag)^*\otimes\CM$ we introduce
$$\Phi_\lambda(x)=e^{<\lambda+\rho_P,H_P( x)>}\Phi(x)$$
and define a representation $I_{P,\lambda}$ by
$$(I_{P,\lambda}(y)\Phi)(x)=e^{<\lambda+\rho_P,H_P(xy)-H_P( x)>}\Phi(xy)$$
or equivalently 
$$(I_{P,\lambda}(y)\Phi)_\lambda(x)=\Phi_\lambda(xy)\ptf$$

 Given two standard parabolic subgroups $\PP$ and $\QQ$
we denote by $\weyl(\ap,\aq)$  the set of elements of minimal length in the Weyl group $\weyl^G$ 
of $\GG$ such that $s(\ap)=\aq$. Recall that two standard parabolic subgroups 
$P$ and $Q$ are said to be associated
if $\weyl(\ap,\aq)$ is non empty. We denote by  $\weyl(\ap,Q)$  
the set of elements of minimal length
such that $s(\ap)\supset\aq$.
Let $w$ denote an element in $\Gamma$ representing $s\in\weyl(\ap,\aq)$
and consider the function $\Psi$ on $X_Q$ defined by the integral, when convergent:
$$\Psi(x)=e^{-<s\lambda+\rho_Q,H_Q( x)>}\int_{N_w} 
e^{<\lambda+\rho_P,H_P(w\mun nx)>}\Phi(w\mun nx)\,dn$$
where $N_w=wN_P w\mun\cap N_Q\bs N_Q$ or, equivalently
$$\Psi_{s\lambda}(x)=\int_{N_w} 
\Phi_\lambda(w\mun nx)\,dn$$
One defines an 
operator $\Mint(s,\lambda)$ 
which intertwines  $I_{P,\lambda}$ and $I_{Q,s\lambda}$ by putting
$$\Psi=\Mint(s,\lambda)\Phi\ptf$$
These operators are products  of local analogues:
$$\Mint(s,\lambda)=\Mint_\infty(s,\lambda)\times\prod_p\Mint_p(s,\lambda)$$
where the product is over prime numbers.

If $\GG=GL(2)$ and $\Phi(x)\equiv1$,
using this product decomposition
 it is easy to compute $\Mint(s,\lambda)\Phi$
when $s$ is the non trivial element in the Weyl group: if
$\lambda=\sigma \rho=\frac{1}{2}\sigma \alpha$ where $\alpha$ the positive root
 and $\sigma\in\CM$ with $\Re(\sigma)>1$
one finds
$$\Mint(s,\lambda)\Phi=m(s,\lambda)\Phi\com{with}
m(s,\lambda)=m_\infty(s,\lambda)\prod_p m_p(s,\lambda)
=\frac{L(\sigma)}{L(1+\sigma)}$$
where $L$ is the complete Riemann Zeta function
$$\Zeta(\sigma)=\Zeta_\infty(\sigma)\prod_p\Zeta(\sigma)
\com{with}
\Zeta_\infty(\sigma)=\pi^{-\sigma/2}\Gamma(\sigma/2)\com{and}\Zeta_p(\sigma)=\frac{1}{1-p^{-\sigma}}\,\,.$$
Similar explicit formulas apply for local intertwining operators acting 
on automorphic forms that are right-invariant under $K_p$ when $\G_p$ is quasi-split and $K_p$
is hyperspecial. For example, according to \cite{Lai} Proposition 3.2, if $\G_p$ is the quasi-split form of $SU(3)$, 
the special unitary group in 3 variables
attached to the unramified quadratic extension of $\QM_p$ at a prime $p\ne2$,
the local factor  $m_p(s,\lambda)$ is
$$m_p(s,\lambda)=\frac{(1-p^{-2(\sigma+1)})(1+p^{-2\sigma-1})}{(1-p^{-2\sigma})(1+p^{-2\sigma})}$$
when $\lambda=\sigma\rho$. For more examples  but in arbitrary rank
see \ref{triv} and \ref{gln} below and \cite{Lai}.

\subsection{Eisenstein Series}

Series of the form
$$E_k(z,s)=\sum_{(c,d)}\frac{\Im(z)^{s-k/2}}{(cz+d)^k|cz+d|^{2s-k}}$$
with $z\in\mathfrak H$  (the upper half plane) and $\Re(s)$ large enough,
 were studied by Hecke, Maa{ss} and Selberg. When $2s=k$ they
  appear in the theory of elliptic modular forms while, when $k=0$, these series
 (or rather their analytic continuation)  
were used by Selberg to  describe the spectral decomposition of $L^2(SL(2,\ZM)\backslash \mathfrak H)$.
 In what follows we shall deal with generalizations of these ones.

Consider $\lambda\in\ap^*\otimes\CM$ and a function $\Phi$ on $X_P$.
One  defines, when convergent, an Eisenstein series by
$$E(x;\Phi,\lambda)=
\sum_{\gamma\in\Gamma_P\bs\Gamma}\Phi_\lambda(\gamma x)
 =\sum_{\gamma\in\Gamma_P\bs\Gamma}
e^{<\lambda+\rho_P,H_P(\gamma x)>}\Phi(\gamma x)\ptf$$
This is a function on $X_G$.  The following lemma is an immediate consequence of \ref{fc}.

\begin{lemma}\label{efc} Eisenstein
series are orthogonal to cusp forms.
\end{lemma}

In fact  Eisenstein
series  will allow to construct
the orthogonal supplement to the space of cusp forms in $\LMC$.
Eisenstein series intertwine representations $I_{P,\lambda}$ and the
right regular representations for functions on $X_G$: if $f$ is a smooth
compactly supported function on $\AG\bs G$ one has
$$E(x;I_{P,\lambda}(f)\Phi,\lambda) =E(x;\Phi,\lambda)\star f\ptf$$
When $\Phi$ is  a
cusp form on $P$, the series defining $E(x;\Phi,\lambda)$
converge for $\Re(\lambda)\in C_P+\rho_P$,
where $C_P$ denotes the projection of the Weyl chamber on $\ap^*$.
This follows from the rapid decay of cusp forms on $X_P$ and of the convergence
of a similar series but where $\Phi=1$  identically.
For such $\lambda$ and $\Phi$ 
the intertwining operator $$\Phi\mapsto \Mint(s,\lambda)\Phi$$ is also given by a convergent integral.
The interplay between Eisenstein series and intertwining operators
shows up when computing their constant terms:
\begin{lemma}\label{cstt} Consider an Eisenstein 
series defined by $\Phi$ cuspidal on $P$.
 The constant term $ E_Q(x;\Phi,\lambda)$ along any parabolic subgroup $Q$ vanishes unless there is an $s$ in the
Weyl group such that $s(\ap)\supset\aq$. In this case one has
$$ E_Q(x;\Phi,\lambda)=\sum_{s  \in \weyl(\ap,Q)}
E^Q(x;\Mint(s,\lambda)\Phi,s\lambda)$$
where $E^Q(x;\Mint(s,\lambda)\Phi,s\lambda)$ is the Eisenstein series on $Q$ defined by $\Mint(s,\lambda)\Phi$.
When $Q$ is associate to $P$ this can be written
$$ E_Q(x;\Phi,\lambda)=\sum_{s  \in \weyl(\ap,\aq)}
e^{<s(\lambda)+\rho_Q,H_Q(x)>}\Mint(s,\lambda)\Phi(x)\ptf$$
\end{lemma}

\subsection{Meromorphic continuation and spectral decomposition}

Before giving some hints toward the proof we state the main theorems
 of chapter 7 (reformulated in adŽlic language in appendix II of \cite{LES}).

\begin{theorem} \label{eiscont}
Assume that $\Phi$ is automorphic on $P$ and that $m\mapsto\Phi(mx)$ 
belongs to the discrete spectrum\footnote{Thanks to Franke's results \cite{Fr}
this hypothesis can be removed: automorphic is enough.}
of $M_P$ for any $x$.  Consider $\lambda\in\ap^*\otimes\CM$.
\pni(i) 
The series defining the Eisenstein series 
$E(x;\Phi,\lambda)$
and the integral defining the
 intertwining operators $M(s,\lambda)$
 are convergent when $\Re(\lambda)$ belongs to
 some translate of the Weyl chamber. 
 \pni (ii) 
 They have a meromorphic continuation on the whole space $\ap^*\otimes\CM$.
 \pni(iii) 
 The operators $M(s,\lambda)$  satisfy functional equations: for  $s\in\weyl(\ap,\aq)$ and 
 $t\in\weyl(\aq,\ar)$
$$\Mint(st,\lambda)=\Mint(s,t\lambda)\Mint(t,\lambda)\ptf$$
 \pni(iv)  Moreover
$$ \Mint(s,-\overline\lambda)^*=\Mint(s,\lambda)\mun\ptf$$
 In particular $M(s,\lambda)$  is unitary when $\lambda$ is purely imaginary.
 \pni(v) 
The automorphic forms $$x\mapsto E(x;\Phi,\lambda)$$ 
obtained  by meromorphic continuation of Eisenstein series satisfy the functional equations: 
$$E(x;\Phi,\lambda)=E(x;\Mint(s,\lambda)\Phi,s\lambda)\ptf$$
 \pni(vi)  Eisenstein series are analytic when $\lambda$ is purely imaginary.
\end{theorem}

Denote by $\HMCPA$ the space of square integrable automorphic forms on $X_P$
and by $\HMCP$ its closure.
The proof of the analytic continuation of general Eisenstein series is entangled with the
proof of the following theorem.

\begin{theorem} \label{disc}
The discrete spectrum  $\LMCD$ is generated by automorphic forms that
appear as residues of 
Eisenstein series constructed from automorphic forms in the cuspidal spectrum of Levi subgroups.
In particular $\LMCD=\HMC_G$.
\end{theorem}

Consider  functions $F_{P}$ on $\iapg$ with values in $\HMCP$ 
such that, for some $\vf$ smooth, $K$-finite and compactly supported on $X_G$ and
for any $\Phi\in\HMCPA$
$$<F_{P}(\lambda),\Phi>_\LMCP=\hvf(\Phi,\lambda)=< \vf,{E(\bullet;\Phi,\lambda)}>_\LMC
=\int_{X_G} \vf(x)\overline{E( x;\Phi,\lambda)}\,d x\ptf$$
The functional equations satisfied by Eisenstein series implies that
$$\hvf(\Phi,\lambda)=\hvf(s\lambda,\Mint(s,\lambda)\Phi)
$$
which is equivalent to
$$<F_{P}(\lambda),\Phi>_\LMCP=<F_{Q}(s\lambda),\Mint(s,\lambda)\Phi>_\LMCQ$$
where $Q$ is the standard parabolic subgroup with Levi subgroup $\Mlev_Q=s(\Mlev_P)$.
This tells us that functions $F_{P}$ satisfy functional equations:
$$F_{Q}(s\lambda)=\Mint(s,\lambda)F_{P}(\lambda)\ptf$$

Let $\AMC$ be the set of  associate classes of  standard
parabolic subgroups and $\AMCP$ a set of representatives of these classes. Let
 $w(P)$ be the order of $\weyl(\ap,\ap)$ and $n(\ap)$  the number of chambers in $(\ap/\ag)^*$. 
Now, consider  $\hat\LMC$ the Hilbert space
 of  collections of measurable functions $F_P$ on $\iapg$
with values in $\HMCP$
satisfying the above functional equations  and are square integrable:
$$\sum_{\PMC\in\AMC}\sum_{P\in\PMC}\frac{1}{ n(\ap) }
\int_{\lambda\in \iapg}
||F_{P}(\lambda)||^2_\HMCP\,d\lambda<\infty\ptf$$
The spectral decomposition is usually formulated as follow (see \cite{A1} or \cite{MW2}
for example).

\begin{theorem}\label{speca} 
 There is a dense subset of function $\vf\in\LMC$ that can be expressed as
$$
\vf(x)=\sum_{\PMC\in\AMC}\sum_{P\in\PMC}\frac{1}{ n(\ap) }\int_{\lambda\in \iapg}
E(x;F_P(\lambda),\lambda)\,d\lambda\ptf
$$
The scalar product  in $\LMC$ may be written 
$$<\vf,\vf>_\LMC=\sum_{\PMC\in\AMC}\sum_{P\in\PMC}\frac{1}{ n(\ap) }
\int_{\lambda\in \iapg}
<F_{P}(\lambda),{F_{P}(\lambda)}>_\HMCP\,d\lambda\ptf$$
This extends to an isomorphism of Hilbert spaces $\hat\LMC\to\LMC$.
 \end{theorem}

A comment is in order as regards the choice of Haar measures.
We choose a Haar measure on $\AG\bs G$ and endow $\Gamma$ with the canonical measure
for a discrete group. This defines the $G$-invariant measure on $X_G$. We take the invariant measure on
$N_P$ such that $\Gamma_P\cap N_P\bs N_P$ has volume 1.  This and the choice of a Haar measure 
$da$ on $\AP/\AG$
produce an invariant measure on $$X_P=\AP\Gamma_P N_P\bs G$$ used to normalize the
scalar product in $\HMCP$. 
The real vector space $\iapg=i(\ap/\ag)^*$ is identified with the Pontryagin dual of $\AP/\AG$ 
via the  pairing $$(\lambda,a)\mapsto e^{<\lambda,H_P(a)>}$$
and  $d\lambda$  is the canonical Pontryagin dual Haar measure of $da$\footnote{%
Observe  that,  instead of Haar measures on groups $\ap/\ag$ and $i(\ap/\ag)^*$ 
related by Pontryagin duality,
surveys  \cite{LB2}  and \cite{A1} use Lebesgue measures 
attached to dual basis in vector spaces $\ap/\ag$ 
 and $(\ap/\ag)^*$  and hence powers of $2\pi i$ show up in their 
 Fourier inversion formulas.}. 
 
 One sometimes uses the following variant of the above theorem.

\begin{theorem} \label{specb}  \pni(i)
For each parabolic subgroup $P$    
one may choose  an orthonormal basis $\B(P)$ of  $L^2_{\mathrm{disc}}(X_P)$ made of automorphic forms
on $P$. 
\pni(ii) There is a dense subset of functions
function $\vf\in\LMC$ that may be writen (with some abuse of notation\footnote{%
Eisenstein series $E(\bullet;\Phi,\lambda)$ do not belong to $\LMC$ and hence
the notation $< \vf,{E(\bullet;\Phi,\lambda)}>_\LMC$  makes sense {\it a priori} only when
the implicit integral is convergent 
which is the case if $\vf$ is compactly supported for exemple. But then, the integral over $\lambda$ may
not be convergent. The isomorphism $\hat\LMC\to\LMC$ 
 in \ref{speca} that extends the correspondence $\vf\mapsto \{F_P\}$
 defined by $<F_P(\lambda),\Phi>=< \vf,{E(\bullet;\Phi,\lambda)}>_\LMC$
 allows to make sense of the statement. Another option is to integrate over (compact)
 balls in $\iapg$ and then take the limits in $\LMC$ when the radii of the balls tend to infinity
(see Theorem VI.2.1 of \cite{MW2}). Except for some simple examples, giving precise
conditions for the convergence of these integrals seems out of reach for the time being.}):
$$\vf(x)=\sum_{\PMC\in\AMC}\sum_{P\in\PMC}\frac{1}{ n(\ap) }\sum_{\Phi\in\B(P)}
\int_{\lambda\in \iapg}< \vf,{E(\bullet;\Phi,\lambda)}>_\LMC
E(x;\Phi,\lambda)\,d\lambda\ptf$$
\pni(iii)
This extends to an isomorphism between Hilbert space $\LMC$ and the Hilbert space of measurable
functions $\psi(\Phi,\lambda)$ such that
$$\sum_{\PMC\in\AMC}\sum_{P\in\PMC}\frac{1}{ n(\ap) }
\sum_{\Phi\in\B(P)}\int_{\lambda\in \iapg}||\psi(\Phi,\lambda)||^2\,d\lambda<\infty\ptf
$$
\pni(iv) In particular, given two smooth $K$-finite compactly supported functions $\vf_1$ and $\vf_2$ on $X_G$ their scalar product
can be written
$$<\vf_1,\vf_2>_\LMC=\sum_{\PMC\in\AMC}\sum_{P\in\PMC}\frac{1}{ n(\ap) }
\sum_{\Phi\in\B(P)}\int_{\lambda\in \iapg}
\hvf_1(\Phi,\lambda)\overline{\hvf_2(\Phi,\lambda)}\,d\lambda\leqno{(\star)}
$$
or equivalently
$$<\vf_1,\vf_2>_\LMC=\sum_{P\in\AMCP}\frac{1}{ w(P)}
\sum_{\Phi\in\B(P)}\int_{\lambda\in \iapg}
\hvf_1(\Phi,\lambda)\overline{\hvf_2(\Phi,\lambda)}\,d\lambda\leqno{(\star\star)}
$$
 with 
$$\hvf_i(\Phi,\lambda)=< \vf_i,{E(\bullet;\Phi,\lambda)}>_\LMC=\int_{X_G} \vf_i(x)\overline{E( x;\Phi,\lambda)}\,d x\ptf$$
\end{theorem}

Assertion (i) is a consequence of Theorem \ref{disc}. 
The equivalence of $(\star)$ and  $(\star\star)$ follows from the following remarks:
 the value of the integral is constant when $P$ varies in a given $\PMC$  since
 intertwining operators $\Mint(s,\lambda)$ are unitary when $\lambda$ is purely imaginary;
it remains to observe that if we denote by $a(\PMC)$  the cardinal of $\PMC$ one has:
$n(\ap)= w(P)a(\PMC)\ptf$

The spectral decomposition allows to express $\LMC$ as a Hilbert direct integral.
Let $\pi$ be an unitary representation of $P=\PP(\ade)$
trivial on $N_P\AG$ and $\lambda\in\iapg$.
We denote by
$\I_P^\G(\pi,\lambda)$
 the right regular representation of $\G$ in the space of  function $\phi$ on $\G$ 
with values in the space of $\pi$ satisfying
$$\Phi(px)=e^{<\lambda+\rho_P,\H_P(p)>}\pi(p)\Phi(x)$$
and that are square integrable on $K$.  The ``parabolically induced''
representation
$\I_P^\G(\pi,\lambda)$ is  unitary.
Define $\FMC_P(\lambda)$ by:
$$\FMC_P(\lambda)=\I_P^\G(L^2_{\mathrm{disc}}(X_P),\lambda)\ptf$$ 
The unitarity of intertwining operators $\Mint(s,\lambda)$ for $\lambda\in\iapg$
shows that for $s\in\weyl(\ap,\aq)$
$$\FMC_P(\lambda)\simeq\FMC_Q(s\lambda)\ptf$$
 Let $\PMC$ be the association class of $P$.
Let $\FMC_\PMC$ be the Hilbert direct integral of 
representations $\FMC_P(\lambda)$:
$$\FMC_\PMC=\int_{\iapg/\weyl(\ap,\ap)}^{\oplus}\FMC_P(\lambda)\,\,d\lambda\ptf$$
Up to isomorphism, it is independent of the choice of $P\in\PMC$.
Theorems \ref{speca}  and \ref{specb} imply that
the right regular representation of $\G$ in
$\LMC=L^2(X_G)$ can be written as the direct sum of direct integrals $\FMC_\PMC$:
\begin{proposition}\label{fmc}
$$\LMC=\bigoplus_{\PMC\in\AMC}\FMC_\PMC=
\bigoplus_{\P\in\AMCP}
\int_{\iapg/\weyl(\ap,\ap)}^{\oplus}\FMC_P(\lambda)\,\,d\lambda\ptf
$$
\end{proposition}

\section{About the proof}

\subsection{Pseudo-Eisenstein series and their scalar product}
Consider a parabolic subgroup $P$ 
and the space  $$X_{P,G}=\AG\Gamma_P N_P\bs G\ptf$$
Denote by $\DMC_P$ the space of $K$-finite functions $\phi$  on $X_{P,G}$  such that 
 $a\mapsto\phi(a\,\bullet\, k)$ is  a compactly supported  function on $\AP/\AG$ with values in a finite 
dimensional space of cuspidal functions on $M_P $ independent of $k\in K$. 
For $\phi\in \DMC_P$ the series
$$\Ei_\phi(x)=\sum_{\gamma\in\Gamma_P\bs\Gamma_G}\phi(\gamma x)$$
are absolutely convergent and, following  \cite{MW2}, we  call them
pseudo-Eisenstein series\footnote{These series are denoted $\hat\phi$
in \cite{LB2}, \cite{A1}  and \cite{Lai}. They appear as series $\theta_\phi$ in \cite{MW2} but there
$\phi$ is the Fourier transform, on the torus, of our $\phi$.}.

Let $\AMC$ be the set of associate classes
of standard parabolic subgroups. For $\PMC\in\AMC$ we denote by $\CMC_\PMC$ the closure of 
the vector space generated by the $\Ei_\phi$ for $\phi\in\DMC_P$
with $P\in\PMC$.
Observe that $\DMC_G=\CMC_G$ is the space of
cusp forms on $\G$.

A first step toward the spectral decomposition is 
a direct sum decomposition indexed by association classes of standard parabolic subgroups.
 This is  Lemma 2 in \cite{LB2} and Proposition II.2.4 in \cite{MW2}. 
 
\begin{proposition} \label{amc}
 One has a direct sum decomposition
$$\LMC=\bigoplus_{\PMC\in\AMC}\CMC_\PMC$$
\end{proposition} 

\begin{proof}
Lemma \ref{fc} and an inductive argument starting with the minimal parabolic subgroup
show that a function on $X_G$ 
which is orthogonal to all $\Ei_\phi$ for all  standard parabolic subgroups, including $G$,
  must vanish. This is Lemma 3.7 p.~55 and its Corollary p.~58 of \cite{LES} (see also Theorem   II.1.12  in \cite{MW2}).
It remains to show that $\CMC_\PMC$ and $\CMC_\QMC$ are orthogonal whenever $\PMC\ne\QMC$. 
 Consider two standard parabolic subgroups $P$ and $Q$ we have either
  \pni (i)  for some element $s$ of the Weyl group $\Mlev_Q\subset s(\Mlev_P)$
  \pni or \pni
 (ii)   for any element $s$ of the Weyl group $\Mlev_Q\cap s(\NN_P)$
 is a non trivial unipotent subgroup.
 \pni 
Consider $P$ and $Q$ that are not associated; then, up to exchanging the role of $P$ and $Q$,
we may assume that (ii) holds. Now consider  functions $\phi$ and $\psi$ 
cuspidal on $P$ and $Q$ respectively. The above formal computation yields
$$<\Ei_\phi,\Ei_\psi>=\int_{X_{P,G}}\phi(x)\overline
{\Ei_{\psi,P}(x)}\,dx $$
where $\Ei_{\psi,P}$ is the constant term along $P$ of $\Ei_{\psi}$. 
 Using  Bruhat decomposition 
 $$\Gamma_P\bs\Gamma/\Gamma_Q\simeq \weyl^P\bs   \weyl^G /  \weyl^Q$$  
 where $\weyl^P$ denotes the Weyl group of  $M_P$,
we see that $\Ei_{\psi,P}$ vanishes since $\psi$ is cuspidal on $Q$. 
\end{proof}

Proposition \ref{amc} reduces the spectral decomposition of $\LMC$
to the spectral decomposition of spaces $\CMC_\PMC$.
To proceed further one needs to compute the scalar product of two
pseudo-Eisenstein series when it does not vanish.
Consider $\phi\in\DMC_P$; then the function $\Phi$ on $X_P\times(\ap/\ag)^*\otimes\CM$   
 given  by
$$\Phi(x,\lambda)=\int_{\AP/\AG}\phi(ax)
e^{-<\lambda+\rho_P,H_P(ax)>}\,da$$
is cuspidal on $\P$ 
in the first variable and analytic of Paley-Wiener type in the second one.
We identify the Pontryagin dual of $\AG\bs\AP$ with $\iapg=i(\ap/\ag)^*$.
By Fourier inversion one recovers $\phi$:
$$\phi(x)=\int_{\lambda\in\lambda_0+\iapg}
\Phi(x,\lambda)e^{<\lambda+\rho_P,H_P(x)>}
\,d\lambda$$ where  $\lambda_0$ is arbitrary and  $d\lambda$ is  the Haar measure
on $\iapg=i(\ap/\ag)^*$ dual,
for Pontryagin duality, to the Haar measure $da$
and hence the pseudo-Eisenstein series $\Ei_\phi$ is equal to an integral of Eisenstein series:
$$\Ei_\phi(x)=\int_{\lambda\in\lambda_0+\iapg}
E(x;\Phi(\bullet,\lambda),\lambda)
\,d\lambda$$
for any $\lambda_0\in C_P+\rho_P$ (which one may also write $\lambda_0>\rho_P$).
We may now give the formula for 
 the scalar product of two pseudo-Eisenstein series which is the key to the spectral decomposition.

\begin{proposition}\label{scp} Consider two associated parabolic subgroups $P$ and $Q$
and two functions $\phi_1\in\DMC_P$ and  $\phi_2\in\DMC_Q$.
 When $\lambda_0\in C_P+\rho_P$ one has, with the notation of \ref{cstt}:
$$<\Ei_{\phi_1},\Ei_{\phi_2}>_\LMC=\sum_{s  \in \weyl(\ap,\aq)}\int_{\lambda\in\lambda_0+\iapg}
<\Mint(s,\lambda)\Phi_1(\bullet,\lambda),{\Phi_2(\bullet,-s\overline\lambda)}>_\LMCP\,d\lambda\ptf
$$ 
\end{proposition}

\begin{proof}This is an immediate consequence of  \ref{fc},
and  \ref{cstt}.
\end{proof}

It is useful to extend the space of cuspidal functions $\Phi(\bullet,\lambda)$ by asking that as a function of $\lambda$
it is holomorphic and rapidly decreasing in vertical strips  $||\lambda||<R$ for some $R$. Their Fourier transform
build a space denoted $\DMC_P(R)$. The above formula still holds provided that moreover
$||\lambda_0||<R$.

The spectral decomposition  is obtained by shifting
the integral to the purely imaginary space $\iapg$ i.e.\removelastskip\  moving $\lambda_0$ to $0$.  
To do this one needs to establish
the meromorphic continuation of intertwining operators and to take into account the residues
that show up.  Analytic estimates are moreover necessary to allow such a contour shift.
One establishes at the same time the analytic continuation and the functional equations
of Eisenstein series.

\subsection{Scalar product of truncated Eisenstein series}\label{formula}

In section 9 of \cite{LB2} Langlands states, without detailed proof,
a formula for the scalar product of two  truncated Eisenstein series induced from cusp forms.
We quote the result with the notation of \cite{LW}.
Let $P$ and $Q$ be two  standard  parabolic subgroups and consider: 
$$\lambda\in\ga_Q^*\otimes\CM\comm{et}\mu\in\ga_R^*\otimes\CM$$
that coincide on $\ga_\G$. One defines
an operator valued function, meromorphic in both variables
whenever intertwining operators are meromorphic,
 $$\omega_{Q|P}^T(\lambda,\mu)=\sum_R\sum_{s\in\weyl(\ga_P,\ga_R)}
 \sum_{t\in\weyl(\ga_Q,\ga_R)}
 e^{<s\lambda-t\mu\,,\,T>}\varepsilon_R(s\lambda-t\mu)
 \Mint(t,-\overline\mu)^*\Mint(s,\lambda)$$
 where  $R$ runs over standard parabolic subgroups associated to $P$ and $Q$.
  The function 
 $\lambda\mapsto\varepsilon_R(\lambda)$ is the inverse of a product of monomials
 $$\varepsilon_R(\lambda)\mun=V_R\prod_{\alpha\in\Delta_R}<\lambda,\check\alpha>
 $$ where the $V_R$ is the volume of the parallelotope generated by $\Delta_R$.  
The function $\omega_{Q|P}^T$
 vanishes unless $P$ and $Q$ are associated. Langlands formula is the following:
 
\begin{theorem}\label{TrSc}
Assume that $\Phi$ and $\Psi$ are cuspidal on $P$ and $Q$. Then, provided $\lambda$ and $\mu$ 
are in the convergence domain for Eisenstein series i.e. the translate by $\rho_P$ (resp. $\rho_Q$)
of the Weyl chamber
$$\int_{X_G}\trunc E(x,\Phi,\lambda)
\overline{\trunc E(x,\Psi,\mu)}dx=<\omega_{Q|P}^T(\lambda,-\overline\mu)\Phi,\Psi>\ptf$$
\end{theorem}

\begin{proof} We refer the reader to section 5.4 of \cite{LW} for
a proof much shorter and elementary 
than the one given in section 4 of \cite{A2}.\footnote
{The proof in \cite{A2}  proceeds directly, expanding the two truncated Eisenstein series: 
this yields a lot of rather complicated terms; 
fortunately cancellations occur but  rely on subtle arguments using 
meromorphic continuation of partial expressions,  growth estimates,
contour shiftings and computation of residues. 
In  \cite{LW} one uses that $\trunc$ is an orthogonal projection
and hence it is equivalent to show that $$\int_{X_G}\trunc E(x,\Phi,\lambda)
\overline{E(x,\Psi,\mu)}dx=<\omega_{Q|P}^T(\lambda,-\overline\mu)\Phi,\Psi>\ptf$$ 
Less and simpler terms have to be dealt with,
no cancellation is needed and the computation is elementary.}
\end{proof}

\begin{theorem}\label{equa} Again $\Phi$ is assumed to be cuspidal.
Assume we know that  intertwining operators are holomorphic in some connected open
set $\mathcal O$ containing the set of $\lambda\in\ap^*\otimes\CM$
such that $\Re(\lambda)\in C_P+\rho_P$
(the translated by $\rho_P$ of the Weyl chamber)
and satisfy the functional equations. Then  Eisenstein series $E(x;\Phi,\lambda)$ 
have a holomorphic continuation in $\mathcal O$ 
 and satisfy the functional equations
$$E(x;\Phi,\lambda)=E(x;\Mint(s,\lambda)\Phi,s\lambda)\ptf$$
\end{theorem}

\begin{proof}
If we have at hand the functional equations for intertwining operators,
one can show
that the singularities arising from zeros of monomials 
 in the denominator of the formula defining $\omega_{Q|P}^T(\lambda,\mu)$
 are cancelled by zeros of the numerator and that
the only possible singularities arise
 from singularities of the intertwining operators $\Mint(\bullet,\bullet)$ 
 (cf. Propositions 1.10.4 and 5.3.3 of \cite{LW}).
Now, if $D$ is a holomorphic differential operator in the variable $\lambda$ the identity 
 \ref{TrSc} tells us that
$$||D\trunc E(\bullet,\Phi,\lambda)||_\LMC^2=D\overline D\,
<\omega_{P|P}^T(\lambda,-\overline\lambda)\Phi,\Phi>_\LMCP\ptf$$
This shows that if $\omega_{Q|P}^T(\lambda,-\overline\lambda)$ is given by a convergent Taylor series
in a neighbourghood of some point $\lambda_0$ then  same is true for the function 
with values in the Hilbert space $\LMC$:
$$\lambda\mapsto \trunc E(\bullet,\Phi,\lambda)\ptf$$
Since both  sides in the formula  \ref{TrSc} are known to be equal in the convergence domain they remain equal in $\mathcal O$. 
Now, if $\vf$ is a smooth compactly supported function on $X_G$
it follows from  \ref{ident} that   for $T$ large enough
$$ \int_{X_G}\vf(x)E(x;\Phi,\lambda)\,dx =\int_{X_G}\vf(x)\trunc E(x;\Phi,\lambda)\,dx\ptf$$
This implies that Eisenstein series 
when considered as distributions on $X_G$:
$$E(\bullet;\Phi,\lambda):\vf\mapsto\int_{X_G}\vf(x)E(x;\Phi,\lambda)\,dx$$
have a holomorphic continuation on $\mathcal O$.
Moreover, since $\Phi$ is automorphic, one may find a
compactly supported smooth function  $f$ on $G$ such that $I_{P,\lambda}(f)\Phi=\Phi$.
But, in the convergence domain, one has
$$E(x;I_{P,\lambda}(f)\Phi,\lambda)=E(x;\Phi,\lambda)\star  f$$
and we get for  $\lambda\in\mathcal O$ an equality of distributions:
$$E(\bullet;\Phi,\lambda)=E(\bullet;\Phi,\lambda)\star  f\ptf$$
This shows that the distributions obtained by
analytic continuation of Eisenstein series are in fact smooth functions.
Using the functional equation for intertwining operators and formula \ref{cstt}, one checks that 
 the constant terms of
 $$ E(x;\Phi,\lambda)\comm{and} E(x;\Mint(s,\lambda)\Phi,s\lambda)$$
 along any proper parabolic subgroup $Q$   are equal
and hence the difference 
$$E(x;\Phi,\lambda)-E(x;\Mint(s,\lambda)\Phi,s\lambda)$$
is cuspidal. But, at the same time, each term is orthogonal to cusp forms
as follows from  \ref{efc}.
The difference must vanish.  This establishes the functional equation for Eisenstein
series in the domain where they are analytic.
\end{proof}

The  theorem  \ref{equa} reduces the proof of the analytic continuation and the functional equations
of Eisenstein series built from cusp forms to the proof of the same properties for intertwining operators.
That the analytic continuation of Eisenstein series yields automorphic forms needs a little more work.

\subsection{Analytic continuation:  cuspidal case}

Consider the case where $\PMC$ is an association class of
maximal standard parabolic subgroups.  
There are two cases: either $\PMC$ has one element $P$ and
$W(\ap,\ap)$ has two elements then we put $P_1=P$
or $\PMC$ has two elements  $ P=P_1\ne P_2=Q$ and
 $W(\ap,\aq)$ has one element. Let 
$$\lambda_{P_i}(z)=z\frac{\alpha_{P_i}}{||\alpha_{P_i}||}  \com{with} z\in\CM$$
 if we denote by $\alpha_{P_i}$
the unique positive root in $\Delta_{P_i}$. Now
 define $\Mint(z)$ by $$\Mint(z)=\Mint(s,\lambda_P(z))$$ when $\PMC$ has one element and
$s$ is the  non trivial element in  $W(\ap,\ap)$
or
$$\Mint(z)=\begin{pmatrix}0&\Mint(s\mun,\lambda_Q(z))\cr\Mint(s,\lambda_P(z))&0\cr
\end{pmatrix}$$
if $s$ belongs to $W(\ap,\aq)$ with $Q\ne P$.
Let  $I=\{1\}$  or $I=\{1,2\}$ according to cases.
Define $$\LMC_\PMC=\bigoplus_{i\in I}\LMC_{P_i}\ptf$$ 
We have to study $\Mint(z)$ and Eisenstein series 
$$E(\bullet,\Phi,z)=\sum_{i\in I} E(\bullet,\Phi_i,z)\ptf$$

\begin{proposition}\label{rkone}
\pni  (i) The functions $\Mint(z)$ and $E(\bullet,\Phi,z)$ are meromorphic for $z\in\CM$.
\pni  (ii)   $\Mint(z)\Mint(-z)=1$
\pni  (iii)   $E(\bullet,\Phi,z)=E(\bullet,\Mint(z)\Phi,-z)$
\end{proposition}

\begin{proof} We sketch an argument borrowed from sections 4 and 6 of \cite{LB2} (see also Lemma 84 in \cite{H}).
Let $r=||\rho_P||$.
Given $\Phi_i$ and $\Psi_i$ in $\LMC_{P_i}$ we define
functions of $z$ with values  in $\LMC_\PMC$
$$\Phi(\bullet,z)=\oplus \Phi_i(\bullet,\lambda_{P_i}(z))
\com{and}\Psi(\bullet,z)=\oplus \Psi_i(\bullet,\lambda_{P_i}(z))\ptf$$ 
 The scalar product of pseudo-Eisenstein series $\Ei_{\phi}$ and $\Ei_{\psi}$ 
attached to $\Phi(\bullet,z)$ and $\Psi(\bullet,z)$ is given by:
$$<\Ei_{\phi},\Ei_{\psi}>_\LMC
=\int_{c+i\RM}<\Phi(\bullet,z),\Psi(\bullet,-\overline z)>_{\LMC_\PMC}
+<\Mint(z)\Phi(\bullet,z),\Psi(\bullet,\overline z)>_{\LMC_\PMC}\,dz$$
provided $c>r$. Here $dz$ is a suitably normalized Haar measure
on $i\RM$.  Consider functions $\Ei_\phi$ with $\phi_i\in\DMC_{P_i}(R)$.
The above formula holds for $$r<c<R\ptf$$
Put $\Phi'_i(\bullet,\lambda)=<\lambda,\lambda>\Phi_i(\bullet,\lambda)$ then, if we define $\Ei_{\phi'}$ using the $\Phi'_i$,
the densely defined linear operator  $A\Ei_\phi=\Ei_{\phi'}$ 
is  essentially self adjoint (unbounded).
Its spectrum is real and its resolvent
$${\mathbf R}(\sigma,A)=(\sigma-A)\mun$$
is a holomorphic function of $\sigma\in\CM$ off the interval $]-\infty,R^2]$.
But since $R$ can be arbitrarily close to $r$ the resolvent is holomorphic off the interval $]-\infty,r^2]$
and $$<{\mathbf R}(\sigma,A)\Ei_{\phi},\Ei_{\psi}>_\LMC$$
equals $$\int_{c+i\RM}\frac{1}{\sigma-z^2}\left(<\Phi(\bullet,z),\Psi(\bullet,-\overline z)>_{\LMC_\PMC}
+<\Mint(z)\Phi(\bullet,z),\Psi(\bullet,\overline z)>_{\LMC_\PMC}\right)\,dz$$
provided $\Re(\sigma)>c^2>r^2$.
For $t\in\CM$, with $\Re(t)\ne \pm c$ while $r<c<R$,
 let
$$r(t,c)=\int_{c+i\RM}\frac{1}{t^2-z^2}\left(<\Phi(\bullet,z),\Psi(\bullet,-\overline z)>_{\LMC_\PMC}
+<\Mint(z)\Phi(\bullet,z),\Psi(\bullet,\overline z)>_{\LMC_\PMC}\right)\,dz\ptf$$
This is a holomorphic function of $t$  for $|\Re(t)|\ne c$ such that
$$r(t,c)=<{\mathbf R}(t^2,A)\Ei_{\phi},\Ei_{\psi}>_\LMC\com{if} \Re(t)>c>r\ptf$$
Suppose now that $$R>c_1>\Re(t)>c>r$$
then shifting the above  integral and evaluating the residue at $z=t$
  we get
$$<{\mathbf R}(t^2,A)\Ei_{\phi},\Ei_{\psi}>_\LMC=s(t)+r(t,c_1)$$ where
$$s(t)=\frac{1}{2t}\left(<\Phi(\bullet,t),\Psi(\bullet,-\overline t)>_{\LMC_\PMC}
+<\Mint(t)\Phi(\bullet,t),\Psi(\bullet,\overline t)>_{\LMC_\PMC}\right)\ptf$$
Hence $s(t)$  has a holomorphic continuation on the subset
$t^2\notin]-\infty,r^2]$.
Now suppose that
  $$\Phi(z)=e^{z^2}\Phi_0\com{and}\Psi(z)=e^{z^2}\Psi_0$$ with $\Phi_0$ and $\Psi_0$ in $\LMC_\PMC$,
  we get
  $$s(t)=\frac{e^{2t^2}}{2t}(<\Phi_0,\Psi_0>_{\LMC_\PMC}+<\Mint(t)\Phi_0,\Psi_0>_{\LMC_\PMC})$$
and hence $\Mint(t)$ has a holomorphic continuation whenever
$t^2\notin]-\infty,r^2]$ i.e. $\Re(t)>0$ and $t\notin]0,r]$.
 The positivity of the inner product  \ref{TrSc}  of a truncated Eisenstein series with itself
 tells us that the hermitian operator
 $$\frac{1}{2\Re(t)}(I-M(t){M(t)^*})
 -\frac{1}{2i\Im(t)}(M(t)-{M(t)^*})$$
 is positive. The continuation for $\Re(t)=0$ but $\!\!\Im(t)\ne0$ and the
functional equation 
$$\Mint(t)\Mint(-t)=I$$ on this line are 
 an easy consequence of this positivity provided one knows the operator remains bounded when
$t=iy+\epsilon$ with $y\in\RM-0$ and $\epsilon>0$ tends to $0$.
This fact will not be established here being too technical: it uses estimates
from section 5 of \cite{LES} (the case  $\GG=GL(2)$ is treated in the Appendix IV
of \cite{LES} pages 326-331. See also Lemma 98 in \cite{H}). It also remains to show that
only a finite number of singular points may show up in the segment $]0,||\rho_P||]$
(see for example sections 6, 7 and 8 of Chapter IV of \cite{H}). By symmetry
one gets the meromorphic continuation to the full complex line.
Meromorphic continuation and functional equation  for  Eisenstein series
now follows from \ref{equa}.
\end{proof}

We may now consider the case of arbitrary rank.
Recall that by definition $N_w=wN_P w\mun\cap N_Q\bs N_Q$.
Now if $s=s_1s_2$ is a factorization of  elements in the Weyl group with length  $\ell(s)=\ell(s_1)+\ell(s_2)$
and if $w_i$ represent the $s_i$, there is an isomorphism
$$N_{w_1}\times N_{w_2}\to N_w$$
which implies that $$M(s,\lambda)=M(s_1,s_2\lambda)M(s_2,\lambda)$$
in the domain of convergence (cf. \cite{Sc}).
Decomposing $s$ into a product of minimal length of simple reflexion  $s_i$ 
we decompose $\Mint(s,\lambda)$ into a product of intertwining operators $\Mint(s_i,\lambda_i)$
and the meromorphic continuation of intertwining operators
is reduced to the rank one case. For more details we refer to section 7 of \cite{LB2} and also 
\cite{Lai} where the case of quasi-split groups is fully treated.
Again, the meromorphic continuation of Eisenstein series induced from cusp forms
and their functional equations now follows from this and \ref{equa} (cf. section 9 of \cite{LB2}).

\subsection{Contour deformation forgetting residues}\label{forg}

Start with the expression \ref{scp}
for the square norm of a pseudo-Eisenstein series
$$<\Ei_\phi,\Ei_\phi>_\LMC=\sum_{s  \in \weyl(\ap,\ap)}\int_{\lambda\in\lambda_0+\iapg}
<\Mint(s,\lambda)\Phi(\bullet,\lambda),{\Phi(\bullet,-s\overline\lambda)}>_\LMCP\,d\lambda\ptf$$
Assume that  the integrand is sufficiently decreasing at infinity so that we may move the integration contour.
By moving $\lambda_0$ to $0$ following a path inside the Weyl chamber
 we may cross singular hyperplanes. Assume for a while that the residues coming from
these crossings vanish. In such a case we get, at least formally,
$$<\Ei_\phi,\Ei_\phi>_\LMC=\sum_{s  \in \weyl(\ap,\ap)}\int_{\lambda\in\iapg}
<\Mint(s,\lambda)\Phi(\bullet,\lambda),{\Phi(\bullet,-s\overline\lambda)}>_\LMCP\,d\lambda$$
which is again equal to
$$=\sum_{s  \in \weyl(\ap,\ap)}\int_{\lambda\in\iapg}
<\Mint(s,\lambda)\Phi(\bullet,\lambda),{\Phi(\bullet,s\lambda)}>_\LMCP\,d\lambda\ptf$$ 
This can be rewritten
$$=\frac{1}{ w(P)}\sum_{s  \in \weyl(\ap,\ap)}\sum_{t  \in \weyl(\ap,\ap)}\int_{\lambda\in\iapg}
<\Mint(s,t\lambda)\Phi(\bullet,t\lambda),{\Phi(\bullet,st\lambda)}>_\LMCP\,d\lambda$$ 
but since
$$\Mint(st,\lambda)=\Mint(s,t\lambda)\Mint(t,\lambda)$$
this is equal to
$$=\frac{1}{ w(P)}\sum_{u  \in \weyl(\ap,\ap)}\sum_{t  \in \weyl(\ap,\ap)}\int_{\lambda\in\iapg}
<\Mint(t,\lambda)\mun\Phi(\bullet,t\lambda),\Mint(u,\lambda)\mun{\Phi(\bullet,u\lambda)}>_\LMCP\,d\lambda\ptf$$ 
Now if we let
$$F_P(\lambda)=\sum_{s  \in \weyl(\ap,\ap)}\Mint(s,\lambda)\mun{\Phi}(\bullet,s\lambda)
$$
we get the 

\begin{proposition}\label{van}
Assume  that the residues coming from
the singularities vanish and that functions are sufficiently decreasing at infinity so that
we may move the integration contour. 
Then
$$<\Ei_\phi,\Ei_\phi>=\frac{1}{ w(P)}
\int_{\lambda\in \iapg}
<F_P(\lambda),{F_P(\lambda)}>\,d\lambda\ptf$$
\end{proposition}
This is the expected contribution coming from $P$ to the spectral decomposition, 
as given in \ref{speca}, when residues do not contribute.

\subsection{The trivial representation as a residue}\label{triv}

The trivial representation will show up when deforming the contour giving the contribution
of the Eisenstein series induced from the constant function $\Phi$ on the 
minimal parabolic subgroup $P_0$ through the multiple residue at $\lambda=\rho$.

Assume   that $\GG$ is a split $\QM$-group and that $K=K_\infty\times\prod_p K_p$ where $K_p$ is hyperspecial
for all prime $p$.
Consider  series $\Ei_\phi$ where $\phi$ is right-$K$-invariant on $X_{P_0}$.
In the expression  \ref{scp} for the scalar product of such pseudo-Eisenstein series is
$$<\Ei_\phi,\Ei_\phi>_\LMC=\sum_{s  \in \weyl^G}\int_{\lambda\in\lambda_0+\iapgo}
<\Mint(s,\lambda)\Phi(\bullet,\lambda),{\Phi(\bullet,-s\overline\lambda)}>_\LMCPO\,d\lambda\ptf$$
The intertwining operators $\Mint(s,\lambda)$
 act on a constant function $\Phi$ by scalars
$$\mint(s,\lambda)=\prod_{\alpha>0,s\alpha<0}\frac{\Zeta(<\lambda,\check\alpha>)}
{\Zeta(1+<\lambda,\check\alpha>)}
$$
where $\Zeta(\sigma)$ is the complete Zeta function:
$$\Zeta(\sigma)=\pi^{-\sigma/2}\Gamma(\sigma/2)\zeta(\sigma)\,\,.$$
Recall that $L(\sigma)$ has a simple pole with residue 1 at $\sigma=1$ and functional equation 
$$\Zeta(\sigma)=\Zeta(1-\sigma)\ptf$$

Observe that the half sum of positive roots $\rho$
 is also the sum of fundamental weights (i.e. the basis dual to the basis of simple co-roots):
 $$\rho=\frac{1}{2}\sum_{\alpha>0}\alpha=\sum_{\alpha\in\Delta_{P_0}}\vp_{\alpha} \ptf
 $$
  Hence $\rho$ is the  intersection point of the affine hyperplanes $<\lambda,\check\alpha>=1$
where $\check\alpha$ runs over simple co-roots.
This implies that the multiple residue at  $\lambda=\rho$
is given by the term indexed by the longest element $s$ in the Weyl group i.e.
the element which sends any positive roots to a negative roots.
Let us denote by $${\prod_{\alpha>0}}'$$ the product over non-simple positive roots and define $V$ by
$$\frac{1}{V}=\frac{\prod'_{\alpha>0}\Zeta(<\rho,\check\alpha>)}
{\prod_{\alpha>0}\Zeta(1+<\rho,\check\alpha>)}\ptf$$
The multiple residue $\mathcal T$ at  $\lambda=\rho$ is given by
$$\mathcal T=\frac{1}{V}<\Phi(\bullet,\rho),\Phi(\bullet,\rho)>_{\LMC_{P_0}}\ptf$$
Now if we denote by $e_G$ the unit element in $G$ we have
$$<\Ei_\phi,1>_\LMC=\Phi(e_G,\rho)$$
if the measure on $G$ is given by $$e^{-<2\rho,H_{P_0}(x)>}dn\, dp\,dk$$
for $x=npk$  where  $dp\,dk$ is normalized so that $$\vol(X_{P_0})=\vol(K)=1$$  while
 $dn$ gives measure 1 to $\Gamma_{P_0}\cap{N_{P_0}}\bs\N_{P_0}$.
This implies that for such a measure on $G$ one has
$$<\Phi(\bullet,\rho),\Phi(\bullet,\rho)>_{\LMC_{P_0}}=|\Phi(e_G,\rho)|^2=|<\Ei_\phi,1>_\LMC|^2$$
and hence $$\mathcal T=\frac{<\Ei_\phi,1>_\LMC<1,\Ei_\phi>_\LMC}{V}\ptf$$
Using this and a resolvent argument Langlands shows in \cite{LB1} that 
$$\vol(X_G)=<1,1>_\LMC=V\ptf$$
In other words, if $\Psi_G$ is  a constant function of norm 1 in $\LMC$
$$\mathcal T=\frac{1}{V}<\Phi(\bullet,\rho),\Phi(\bullet,\rho)>_{\LMC_{P_0}}
=<\Ei_\phi,\Psi_G>_\LMC<\Psi_G,\Ei_\phi>_\LMC$$
and hence $\mathcal T$ is the contribution to the spectral decomposition
of the trivial representation.

This equation is the starting point for proving  Weil's conjecture which says that 
the  Tamagawa number $\tau(G)=1$  when  $G$ simply connected. The conjecture was 
proved by Langlands for split groups \cite{LB1}, and extended by
K.F. Lai to arbitrary quasi-split groups  \cite{Lai}. As Langlands had expected,
the proof of the general case, due to Kottwitz \cite{K}, reduces to Lai's result modulo 
the stabilization of the Trace Formula in a special case.

\subsection{The general case}

It remains to deal with the most difficult part of the proof: to take into account  
more general residues that show up when moving the contour. 
This is taken care of in Chapter 7 of \cite{LES} and in Chapters V and VI in \cite{MW2}.
A nice introduction is given
by J. Arthur in \cite{A1} and since we could not do any better we refer the reader to this survey.
Let us simply say that
besides serious analytic difficulties the main obstacles come from the following facts: 
one does not know {\it a priori} 
the location of singularities of the intertwining operators. In  the simplest cases
intertwining operators are quotients of Riemann Zeta functions and, if their poles 
are known, their zeros are mysterious and may produce singularities. 
In the general case this is worse since one has to control intertwining operators
which involve the most general $L$-functions attached to automorphic forms by Langlands and
very little is known about their singularities.  Moreover, as will be seen in the  $GL(3)$ example
(in the proof of \ref{vanC}), 
 many parasitic residues do cancel. A direct argument, 
like the one we shall use in examples,  would not be tractable in the general case and a subtle
induction process is necessary.

One has thus established two orthogonal decompositions: an easy one
 in \ref{amc}:
$$\LMC=\bigoplus_{\PMC\in\AMC}\CMC_\PMC$$
and a  much deeper one \ref{fmc}.
The various  projectors commutes with each other. Putting 
$$\GMC_\PMC^\QMC=\CMC_\PMC\cap\FMC_\QMC$$
we get a finer decomposition:
\begin{proposition}\label{fmcb}
$$\LMC=\bigoplus_{\PMC\in\AMC}\mskip 5mu
\bigoplus_{\{\QMC\in\AMC|\PMC<\QMC\}}\GMC_\PMC^\QMC\ptf$$
Here $\PMC<\QMC$ means that given $P\in\PMC$ there is a $Q\in\QMC$
such that $P\subset Q$.
\end{proposition}
A representation occuring in $\GMC_\PMC^\QMC$ is induced from a 
representation in the discrete spectrum of $Q\in\QMC$ which in turn appears 
via a residue from an induced representation of a cuspidal representation of $P\in\PMC$
(i.e. a representation of the cuspidal spectrum on $X_P$).

\section{Examples of spectral decomposition}

When $\GG=SL(2)$  the spectral decomposition is due to Selberg \cite{S1}
(see [Ku] for a detailed treatment in the classical language).
A treatment in the adelic setting is given in \cite{God}.
For the quasi-split form of $U(3)$ attached to a quadratic extension $E/F$,
the non-cuspidal discrete spectrum is described in section 13.9 of \cite{Rog}.

The spectral decomposition for $GL(n)$ when $n=2$ or $3$,  restricted to the $K$-invariant vectors,
is treated below following section 10 of \cite{LB2}.
 We refer to appendix III of \cite{LES} 
for a similar study 
for groups $G_2$ and $SL(4)$ where new intrinsic difficulties show up.

\subsection{Some notation}\label{gln}

Let $\LMC^K$ be the space of $K$-invariant functions
in $\LMC$.  This is a representation space for the spherical Hecke algebra i.e.
the convolution algebra of compactly supported left and right $K$-invariant functions.
This is a commutative algebra and hence irreducible representations have dimension 1.
Proposition \ref{amc} yields a first decomposition 
$$\LMC^K=\sum_{\PMC\in\AMC}\CMC_{\PMC}^K$$
where $\CMC_{\PMC}^K$ is the space of $K$-invariant functions in $\CMC_{\PMC}$.
This is the closure of the space generated by pseudo-Eisenstein series 
constructed from $K$-invariant cuspidal functions $$x\mapsto\Phi(x,\lambda)$$ 
on $X_{P}$
 for $P\in\PMC$. 
When $\PMC=\{P_0\}$ one has
 $$X_{P_0}=\APO\Gamma_{P0} N_{P_0}\bs G=P_0\bs G$$ and
 $G=P_0 K$. In such a case functions  $x\mapsto\Phi(x,\lambda)$ are simply constant on $G$ and
the scalar product in $\LMCPO$ is given by
$$<\Phi(\bullet,\lambda),\Phi(\bullet,\overline\lambda)>_\LMCPO=\Phi(e_G,\lambda)\overline{\Phi(e_G,\overline\lambda)}\ptf$$
 The intertwining operators $\Mint(s,\lambda)$ act on such functions by scalars 
 $$\mint(s,\lambda)=\prod_{\alpha>0,s\alpha<0}\frac{\Zeta(<\lambda,\check\alpha>)}
{\Zeta(1+<\lambda,\check\alpha>)}\ptf$$

\subsection{The case $\GG=GL(2)$}

We now consider $\GG=GL(2)$.
The minimal parabolic subgroup $P_0$ can be taken to be the subgroup of upper
triangular matrices 
$$P_0=\begin{pmatrix}\star&\star\cr0&\star\cr
\end{pmatrix}\ptf$$
 In $G_\infty$ the maximal compact subgroup $K_\infty=O(2,\RM)$ while in $G_p$ we take 
$K_p=GL(2,\ZM_p)$
where $\ZM_p$ is the ring of $p$-adic integers.
When $s$ is the non trivial element in the Weyl group
 $$\mint(s,\lambda)=\frac{\Zeta(<\lambda,\check\alpha>)}
{\Zeta(1+<\lambda,\check\alpha>)}\ptf$$
We start from 
$$<\Ei_\phi,\Ei_\phi>_\LMC=\int_{\lambda\in\lambda_0+\Lambda_{P_0}}
\Phi(e_G,\lambda)
\overline{{\Phi(e_G,-\overline\lambda)}}+\mint(s,\lambda)\Phi(e_G,\lambda)
\overline{{\Phi(e_G,-s\overline\lambda)}}
\,d\lambda\ptf$$
Moving the contour from $\lambda_0$ to $0$ we have to take into account the residue
of $\mint(s,\lambda)$ at $\lambda=\rho=\alpha/2$ and
we get
$$<\Ei_\phi,\Ei_\phi>_\LMC=\int_{\lambda\in\Lambda_{P_0}}
\left(\Phi_0(\lambda)
\overline{{\Phi_0(\lambda)}}+\mint(s,\lambda)\Phi_0(\lambda)
\overline{{\Phi_0(s\lambda)}}\right)
\,d\lambda+\frac{1}{\Zeta(2)}\Phi_0(\rho)\overline{\Phi_0(\rho)}
$$
where
$$\Phi_0(\lambda)=\Phi(e_G,\lambda) \ptf$$
Here $\Lambda_{P_0}$ is of dimension 1. This yields the
\begin{proposition} For $\GG=GL(2)$ the spectral decomposition
of the space generated by $K$-invariant pseudo-Eisenstein series on $P_0$ is given by:
$$<\Ei_\phi,\Ei_\phi>_\LMC=\frac{1}{ 2}\int_{\lambda\in \Lambda_{P_0}}
<\Ei_\phi,E(\bullet;\Psi_0,\lambda)>_\LMC
\overline{<\Ei_\phi,E(\bullet;\Psi_0,\lambda)>_\LMC}\,d\lambda\mskip150mu$$
$$\mskip 350mu
+<\Ei_\phi,\Psi_G>_\LMC<\Psi_G,\Ei_\phi>_\LMC\ptf$$
Here $\Psi_0$ is a constant function on $G$ normalized so that $<\Psi_0,\Psi_0>_{\LMC_{P_0}}=1$
and
$\Psi_G$ is a constant function on $G$ normalized so that $<\Psi_G,\Psi_G>_\LMC=1$.
The orthogonal complement of $\CMC_{P_0}^K$ in $\LMC^K$ is $\CMC_G^K=\LMCC^K$.
\end{proposition}

\subsection{The case $\GG=GL(3)$. Preparatory material}
When  $\GG=GL(3)$
our minimal parabolic subgroup $P_0$ is the subgroup of upper
triangular matrices and there are two associated standard maximal parabolic subgroups $P_1$ and $P_2$
$$
P_0=\begin{pmatrix}\star&\star&\star\cr0&\star&\star\cr0&0&\star\cr
\end{pmatrix}\comm{}
P_1=\begin{pmatrix}\star&\star&\star\cr\star&\star&\star\cr0&0&\star\cr
\end{pmatrix}\comm{}
P_2=\begin{pmatrix}\star&\star&\star\cr0&\star&\star\cr0&\star&\star\cr
\end{pmatrix}
$$
The maximal compact subgroup  $K_\infty$ in $G_\infty$ is $O(3,\RM)$ while $K_p=GL(3,\ZM_p)\ptf$
 We denote by ${\alpha_1}$ and ${\alpha_2}$ the two simple roots
and by $\vp_1$ and $\vp_2$ the corresponding fundamental weights.
 Recall that $\rho$ the half sum of positive roots is such that
 $$\rho=\frac{1}{2}({\alpha_1}+{\alpha_2}+{\alpha_3})={\alpha_1}+{\alpha_2}
 ={\alpha_3}=\vp_1+\vp_2$$ and hence
 $$<\rho,\check{\alpha_1}>=<\rho,\check{\alpha_2}>=1\com{while}<\rho,\check\rho>=2\ptf$$
Besides  the three symmetries $s_i$ defined by $\alpha_i$ 
the two other non trivial elements in the Weyl group are rotations
$$r_1=s_1 s_2\com{and}r_2=s_2 s_1\ptf$$

Let $\delta_i=\frac{1}{2}\alpha_i$ and denote by $\sigma_{ij}$ the unique element in the Weyl group which
sends $\delta_i$ to $-\delta_j$. Consider vectors so that the angle between $\delta_i$ and $e_i$ is $\pi/2$:
$$e_1=-\vp_2\com{,}e_2=\vp_1\com{and}e_3=z(\vp_1-\vp_2)\ptf$$
We shall need the:
\begin{lemma}\label{eij} When
 $ \lambda_i=\delta_i+ze_i$ with $\overline z=-z\in\CM$ we have
$$\overline{-\sigma_{ij}(\lambda_i)}=\lambda_j$$
\end{lemma} 
\begin{proof}
One has simply to observe that
$\sigma_{ij}(e_i)=e_j$.  This is obvious when $i=j$ and also when $\sigma_{ij}$ is a rotation.
It remains to check it when $\sigma_{ij}=s_3$ and $i<j$ which occurs only when $i=1$ and $j=2$
but $s_3(\vp_1)=-\vp_2$.
\end{proof}

We have
$$\mint(s_1,\lambda)=\frac{\Zeta(<\lambda,\check{\alpha_1}>)}
{\Zeta(1+<\lambda,\check{\alpha_1}>)}\com{,}\mint(s_2,\lambda)=\frac{\Zeta(<\lambda,\check{\alpha_2}>)}
{\Zeta(1+<\lambda,\check{\alpha_2}>)}
$$
$$\mint(r_1,\lambda)=
\frac{\Zeta(<\lambda,\check{\alpha_2}>)}
{\Zeta(1+<\lambda,\check{\alpha_2}>)}
\frac{\Zeta(<\lambda,\check\rho>)}
{\Zeta(1+<\lambda,\check\rho>)}\!\!\!\!
\com{,}\!\!\!\mint(r_2,\lambda)=\frac{\Zeta(<\lambda,\check{\alpha_1}>)}
{\Zeta(1+<\lambda,\check{\alpha_1}>)}
\frac{\Zeta(<\lambda,\check\rho>)}
{\Zeta(1+<\lambda,\check\rho>)}
$$
and since $s_3=s_1s_2s_1$ 
$$\mint(s_3,\lambda)=\frac{\Zeta(<\lambda,\check{\alpha_1}>)}
{\Zeta(1+<\lambda,\check{\alpha_1}>)}
\frac{\Zeta(<\lambda,\check{\alpha_2}>)}
{\Zeta(1+<\lambda,\check{\alpha_2}>)}
\frac{\Zeta(<\lambda,\check\rho>)}
{\Zeta(1+<\lambda,\check\rho>)}
\ptf$$
The singularities of functions $\mint(s,\lambda)$  in the domain 
$$\mathcal R_s=\{\lambda\,\,|\,\,\Re(<\lambda,\check\alpha_i>)\ge0\com{for}\alpha_i>0\com{and}s(\alpha_i)<0\}$$
 are  along the affine subspaces 
$$\Lambda_i=\{\lambda\,|<\lambda,\check\alpha_i>=1\}\ptf
$$ In fact,   for $\lambda\in\mathcal R_s$ factors in the numerators are holomorphic except for poles
whenever  $<\lambda,\check\alpha_i>=1$ or $0$ for some $i$, while
 denominators are holomorphic and non zero whenever $<\lambda,\check\alpha_i>\ne0$;
moreover $\mint(s,\lambda)$ is holomorphic when $(<\lambda,\check\alpha_i>)=0$
for some $i$ since singularities of numerator and denominator cancel.

The subspaces $\Lambda_i$ are  the sets of $\lambda_i=\delta_i+ze_i$ with $z\in\CM$.
The residue of  $\mint(\sigma_{ij},\lambda)$
at $\lambda=\lambda_i=\delta_i+ze_i$
can be written
$$
\frac{1}{\Zeta(2)}n_{ij}(z)
=\frac{1}{\Zeta(2)}
n(\sigma_{ij},\lambda_i)$$ 
The  matrix $\mathcal N(z)$ with entries  $n_{ij}(z)=n(\sigma_{ij},\lambda_i)$ is given by:
$$
\mathcal N(z)=\begin{pmatrix}
1&
n(s_3,\lambda_1)
&
n(r_2,\lambda_1)
\cr
n(s_3,\lambda_2)
&
1&
n(r_1,\lambda_2)
\cr
n(r_1,\lambda_3)
&
n(r_2,\lambda_3)
&n(s_3,\lambda_3)
\cr\end{pmatrix}
$$
Now, taking into account that 
\renewcommand{\arraystretch}{1.9} 
$$\Zeta(1+<\lambda_1,\check{\alpha_2}>)=\Zeta(<\lambda_1,\check\rho>)
\com{and}\Zeta(1+<\lambda_2,\check{\alpha_1}>)=\Zeta(<\lambda_2,\check\rho>)
$$ we see that  
$$
n(s_3,\lambda_1)=\frac{\Zeta(<\lambda_1,\check{\alpha_2}>)}{\Zeta(1+<\lambda_1,\check{\alpha_2}>)}
\frac{\Zeta(<\lambda_1,\check\rho>)}{\Zeta(1+<\lambda_1,\check\rho>)}
=\frac{\Zeta(<\lambda_1,\check{\alpha_2}>)}{\Zeta(1+<\lambda_1,\check\rho>)}
$$
 and a similar cancellation occurs for $n(s_3,\lambda_2)$. This shows that 
 $$\mathcal N(z)=
\begin{pmatrix}
1&
\frac{\Zeta(<\lambda_1,\check{\alpha_2}>)}{\Zeta(1+<\lambda_1,\check\rho>)}
&
\frac{\Zeta(<\lambda_1,\check\rho>)}{\Zeta(1+<\lambda_1,\check\rho>)}
\cr
\frac{\Zeta(<\lambda_2,\check{\alpha_1}>)}{\Zeta(1+<\lambda_2,\check\rho>)}
&1&
\frac{\Zeta(<\lambda_2,\check\rho>)}{\Zeta(1+<\lambda_2,\check\rho>)}
\cr

\frac{\Zeta(<\lambda_3,\check{\alpha_1}>)}{\Zeta(1+<\lambda_3,\check{\alpha_1}>)}
&
\frac{\Zeta(<\lambda_3,\check{\alpha_2}>)}{\Zeta(1+<\lambda_3,\check{\alpha_2}>)}
&

\frac{\Zeta(<\lambda_3,\check{\alpha_1}>)}{\Zeta(1+<\lambda_3,\check{\alpha_1}>)}
\frac{\Zeta(<\lambda_3,\check{\alpha_2}>)}{\Zeta(1+<\lambda_3,\check{\alpha_2}>)}
\cr
\end{pmatrix}$$
Since $<\delta_1,\check\alpha_2>=<\delta_2,\check\alpha_1>=-\frac{1}{2}$
and  $<\delta_i,\check\rho>=\frac{1}{2}$
we get 
$$\mathcal N(z)=\begin{pmatrix}
1&
\frac{\Zeta(-z-\frac{1}{2})}{\Zeta(-z+\frac{3}{2})}
&
\frac{\Zeta(-z+\frac{1}{2}))}{\Zeta(-z+\frac{3}{2}))}
\cr
\frac{\Zeta(z-\frac{1}{2})}{\Zeta(z+\frac{3}{2})}
&
1&
\frac{\Zeta(z+\frac{1}{2}))}{\Zeta(z+\frac{3}{2}))}
\cr
\frac{\Zeta(z+\frac{1}{2}))}{\Zeta(z+\frac{3}{2}))}
&
\frac{\Zeta(-z+\frac{1}{2}))}{\Zeta(-z+\frac{3}{2}))}
&
\frac{\Zeta(z+\frac{1}{2}))}{\Zeta(z+\frac{3}{2}))}
\frac{\Zeta(-z+\frac{1}{2}))}{\Zeta(-z+\frac{3}{2}))}
\cr
\end{pmatrix}
$$

\begin{lemma} \label{nz}
The matrix $\mathcal N(z)$ is of rank one {and} $n_{ij}( z)=n_{ji}(-z)$. \end{lemma}
\begin{proof}
That $\mathcal N(z)$ is of rank one 
can be checked using its explicit expression and the functional equation for $\Zeta(\sigma)$.
\end{proof}
\begin{lemma} \label{run}
One has $$n_{ij}(z)=n_{ik}(z)\overline{n_{jk}(z)}\com{if} \overline z=-z
\com{for $k=1$ or $2$}\ptf$$
\end{lemma}
\begin{proof}Since $\mathcal N(z)$ is of rank one
there exists functions $c_{ij}(z)$ 
such that for any $i$, $j$ or $k$
$$c_{ik}(z)n_{kj}(z)=n_{ij}(z)\com{and}c_{ik}(z)c_{kj}(z)=c_{ij}(z)\ptf$$
In particular
$$c_{ik}(z)c_{ki}(z)=1\ptf$$
But since $n_{kk}=1$ for $k=1$ or $2$ we have
$c_{ki}(z)n_{ik}(z)=1$
and  hence
$$c_{ik}(z)=n_{ik}(z)\com{for $k=1$ or $2$}\ptf$$
This shows that
$$n_{ij}(z)=n_{ik}(z)n_{kj}(z)$$
for $k=1$ or $2$ and the lemma follows since
$$n_{jk}( z)=\overline{n_{kj}(z)}\com{if} \overline z=-z\ptf$$
\end{proof}
\subsection{Contour deformation for $\CMC_{P_0}^K$}\label{dr}

We start from 
$$<\Ei_\phi,\Ei_\phi>_\LMC=\sum_{s  \in \weyl(\apo,\apo)}\int_{\lambda\in\lambda_0+\Lambda_{P_0}}
\mint(s,\lambda)\Phi(e_G,\lambda)
\overline{{\Phi(e_G,-s\overline\lambda)}}
\,d\lambda\ptf$$
When we move $\lambda_0$ to $0$,
 ignoring for a while the singularities on affine lines $<\lambda,\check\alpha_i>=1$
 of the intertwining operators, we get an integral over a space of dimension 2:
$$A=\sum_{s  \in \weyl(\apo,\apo)}\int_{\lambda\in\Lambda_{P_0}}
\mint(s,\lambda)\Phi_0(\lambda)
\overline{{\Phi_0(s\lambda)}}
\,d\lambda\ptf$$
with $\Phi_0(\lambda):=\Phi(e_G,\lambda)$.
Here $\Lambda_{P_0}$ is of dimension 2.

We get, as a particular case of  \ref{van},  the expected contribution to
$<\Ei_\phi,\Ei_\phi>_\LMC$
coming from $P_0$ to the spectral decomposition, 
when the contribution of residues is omitted. It will be denoted by $A$

\begin{proposition}\label{vanA}
$$A=\frac{1}{ 6}\int_{\lambda\in \Lambda_{P_0}}
<\Ei_\phi,E(\bullet;\Psi_0,\lambda)>_\LMC
\overline{<\Ei_\phi,E(\bullet;\Psi_0,\lambda)>_\LMC}\,d\lambda$$
where $\Psi_0$ is a constant function on $G$ normalized so that $<\Psi_0,\Psi_0>_{\LMC_{P_0}}=1$
\end{proposition}

Now, consider the residue on affine lines $<\lambda,\check\alpha_i>=1$ and move the contour to 
 $\Lambda_i$ where
$$\Lambda_i=\{\lambda_i  \, | \,<\lambda_i,\check\alpha_i>=1, \Re(\lambda_i)=\delta_i\}
$$
are  affine subspaces of real dimension 1: $ \lambda_i=\delta_i+ze_i$ with $\overline z=-z\in\CM$.
 Ignoring the possible {\it ``double residues''} that will be treated later on 
  (see \ref{vanC}) and thanks to \ref{eij} we see that 
  $$\overline{-\sigma_{ij}(\lambda_i)}=\lambda_j$$
and we get a contribution of the form:
 $$B=\frac{1}{\Zeta(2)}\sum_{i=1}^{3}\sum_{j=1}^{3} \int_{z\in i\RM}
n_{ij}(z)\Phi_i(z)
\overline{{\Phi_j(z)}}\,dz$$ 
if  $$\Phi_i(z)=\Phi(e_G,\lambda_i )=\Phi(e_G,\delta_i+ze_i )\ptf$$
Now, let us denote by $\Psi_0$ the function identically equal to 1. It yields 
an element of norm 1 in $\LMCPO$ since, by assumption, $\vol(K)=1$. Let
$E_i(x;\Psi_0,ze_i)$ be the residue of $E(x;\Psi_0,\lambda)$ on the affine complex line
$\lambda=\delta_i+ze_i$; its constant term along $P_0$ is given by
$$\int_{\Gamma\cap N_0\bs N_0}  E_i(nx;\Psi_0,ze_i)\,dn=\frac{1}{\Zeta(2)}
\sum_{j=1}^{3}e^{<\sigma_{ij}(\delta_i+ze_i)+\rho,H_0(x)>}
n_{ij}(z)\Psi_0(x)$$
which, thanks to \ref{eij}, is  equal to
$$\frac{1}{\Zeta(2)}
\sum_{j=1}^{3}e^{<-\delta_j+ze_j+\rho,H_0(x)>}
n_{ij}(z)\Psi_0(x)\ptf$$
Now
$$<\Ei_\phi,E_i(\bullet;\Psi_0,ze_i)>=\frac{1}{\Zeta(2)}\sum_{j=1}^{3}
\int_{X_{P_0}}\phi(x) e^{<-\delta_i-ze_j+\rho,H_0(x)>}
\overline{ n_{ij}(z)\Psi_0(x)}\,dx
$$ Since $\Psi_0(x)\equiv 1$ and $\overline{n_{ij}(z)}=n_{ji}(z)$ {if} $\overline z=-z$ we have
$$<\Ei_\phi,E_i(\bullet;\Psi_0,ze_i)>=\frac{1}{\Zeta(2)}\sum_{j=1}^{3}n_{ji}(z) \Phi_j(z)
\ptf$$
Recall that
 $$B=\frac{1}{\Zeta(2)}\sum_{i=1}^{3}\sum_{j=1}^{3} \int_{z\in i\RM}
n_{ij}(z)\Phi_i(z)
\overline{{\Phi_j(z)}}\,dz
\ptf$$
Thanks to \ref{run} this can be rewritten
$$ B=\frac{1}{\Zeta(2)}\sum_{i=1}^{3}\sum_{j=1}^{3} \int_{z\in i\RM}
n_{i1}(z)\overline{n_{j1}(z)}\Phi_i(z)
\overline{{\Phi_j(z)}}\,dz$$
and altogether we get
$$B=\Zeta(2)\int_{i\RM}
<\Ei_\phi,E_1(\bullet;\Psi_0,ze_i)>
\overline{<\Ei_\phi,E_1(\bullet;\Psi_0,ze_i)>}\,dz\ptf$$
Observe that we could have chosen $E_2$ instead of $E_1$.
Let $\Psi_1$ to be a constant function of norm 1 in $\LMC_{P_1}$.
We may take
$$\Psi_1(x)=\frac{\Psi_0(x)}{\sqrt{\Zeta(2)}}\equiv\frac{1}{\sqrt{\Zeta(2)}}$$
and hence the contribution $B$ of simple residues 
to the spectral decomposition for the space of $K$-invariant vectors is given by
\begin{proposition}\label{vanB}
$$B=\int_{\lambda_1\in\Lambda_{P_1}}
<\Ei_\phi,E_1(\bullet;\Psi_1,\lambda_1)>_\LMC
\overline{<\Ei_\phi,E_1(\bullet;\Psi_1,\lambda_1)>_\LMC}\,\,d\lambda_1$$
 where $\Lambda_{P_1}$ is the line $ze_1$ with $z$ imaginary.
\end{proposition}

Using the associate
parabolic subgroup $P_2$ instead of $P_1$ 
one would get the same expression for $B$ but with $E_2$ in place of $E_1$.

We have now to list elements $s$ in the Weyl group and values of $$\lambda=a\vp_1+b\vp_2$$  such that
$\mint(s,\lambda)$ has a pole in each variable $a$ and $b$ and compute the double-residues.  
There are 5 possibilities.
The function $\mint(r_1,\lambda)$ (resp. $\mint(r_2,\lambda)$) has a pole in each variable
 when $\lambda=\vp_2$ (resp. $\lambda=\vp_1$)
and $\mint(s_3,\lambda)$ has also a pole in each variable when $\lambda=\vp_2$ or $\lambda=\vp_1$.
The double-residues are $$\frac{1}{\Zeta(2)\Zeta(2)}\com{for} \mint(r_1,\lambda)\com{or}\mint(r_2,\lambda)\ptf$$
Since $r_1(\vp_2)=-\vp_1$ and
$r_2(\vp_1)=-\vp_2$  
 we get as contributions from $r_1$ and $r_2$ 
$$\frac{1}{\Zeta(2)\Zeta(2)}\Phi(e_G,\vp_2)\overline{\Phi(e_G,\vp_1)}\com{and}
 \frac{1}{\Zeta(2)\Zeta(2)} \Phi(e_G,\vp_1)\overline{ \Phi(e_G,\vp_2)} $$
Now observe that $$\lim_{\lambda\to\vp_i}
\frac{\Zeta(<\lambda,\check{\alpha_i}>)}
{\Zeta(1+<\lambda,\check{\alpha_i}>)}=-1$$  for $i=1$ or 2.
This shows that the double residues for $\mint(s_3,\lambda)$ at $\lambda=\vp_i$
are equal to
$$\frac{-1}{\Zeta(2)\Zeta(2)}\ptf$$ Since 
$s_3(\vp_2)=-\vp_1$ and $s_3(\vp_1)=-\vp_2$
the double residues of
$$\mint(s_3,\lambda)\Phi(e_G,\lambda)
\overline{{\Phi(e_G,-s\overline\lambda)}}$$  are
$$\frac{-1}{\Zeta(2)\Zeta(2)}\Phi(e_G,\vp_2)\overline{\Phi(e_G,\vp_1)}\com{and}
 \frac{-1}{\Zeta(2)\Zeta(2)}\Phi(e_G,\vp_1)\overline{ \Phi(e_G,\vp_2)} \ptf$$
These two contributions of $s_3$ cancel those coming from $r_1$ and $r_2$ and hence
the double residues at $\lambda=\vp_i$
do not contribute to the spectral decomposition. The fifth double residue occurs for
  $\mint(s_3,\lambda)$ at  $\lambda=\rho$ and 
 is given by $$\frac{1}{\Zeta(2)\Zeta(3)}$$
so that, altogether the contribution of the double residues is
$$C=\frac{1}{\Zeta(2)\Zeta(3)}\Phi(e_G,\rho)\overline{\Phi(e_G,\rho)}\ptf$$
Using that
$$<1,1>_\LMC=\vol(X_G)=\Zeta(2)\Zeta(3)$$
when the volume is computed using the Haar measure described above,
we get the
\begin{proposition}\label{vanC}
$$C=<\Ei_\phi,\Psi_G>_\LMC<\Psi_G,\Ei_\phi>_\LMC$$
if $\Psi_G$ is  a constant function  of norm 1 in $\LMC$.
\end{proposition}

Summing up the contributions $A$, $B$ and $C$ given in \ref{vanA}, \ref{vanB} and \ref{vanC},
we have the
\begin{proposition}. For $\GG=GL(3)$ the spectral decomposition
of the space $\CMC_{P_0}^K$  is given by:
$$<\Ei_\phi,\Ei_\phi>_\LMC=\mskip550mu $$
$$\mskip-100mu\frac{1}{ 6}\int_{\lambda\in \Lambda_{P_0}}
<\Ei_\phi,E(\bullet;\Psi_0,\lambda)>_\LMC
\overline{<\Ei_\phi,E(\bullet;\Psi_0,\lambda)>_\LMC}\,\,d\lambda$$
$$\mskip110mu+\int_{\lambda_1\in \Lambda_{P_1}}
<\Ei_\phi,E_1(\bullet;\Psi_1,\lambda_1)>_\LMC
\overline{<\Ei_\phi,E_1(\bullet;\Psi_1,\lambda_1)>_\LMC}\,\,d\lambda_1$$
$$\mskip450mu+<\Ei_\phi,\Psi_G>_\LMC<\Psi_G,\Ei_\phi>_\LMC\ptf$$
\end{proposition}

\subsection{Other contributions to $\LMC^K$}

Recall that 
$$\LMC^K=\CMC_{P_0}^K\oplus\CMC_{\PMC}^K\oplus\CMC_G^K$$
 where $\PMC$ is the association class of maximal  standard parabolic subgroups.
Since $$\CMC_G^K=\LMCC^K$$
 it remains to analyze the space 
$\CMC_{\PMC}^K$ generated
by  pseudo-Eisenstein series induced by $K$-invariant cusp form on the maximal parabolic subgroups $P_i\in\PMC$.
Given $\phi_i\in\DMC_{P_i}$, the scalar
product  is given by
$$<\Ei_{\phi_i},\Ei_{\phi_j}>_\LMC=\sum_{s  \in \weyl(\goth a_{P_i},\goth a_{P_j})}
\int_{\lambda_i\in\lambda_0+\Lambda_{P_i}}
<\Mint(s,\lambda_i)\Phi_1(\bullet,\lambda_i),{\Phi_j(\bullet,-s\overline\lambda_i)}>_{\LMC_{P_i}}\,d\lambda_i\ptf
$$ 
We observe that $\weyl(\goth a_{P_i},\goth a_{P_j})$ is reduced to a single element.
When $i=j=1$ the deformation of the contour meets no singular point and yields a contribution of the
form 
$$<\Ei_{\phi_1},\Ei_{\phi_1}>_\LMC=\sum_{\Psi\in\B_{\mathrm{cusp}}(P_1)}\int_{\lambda_1\in\Lambda_{P_1}}
<\Ei_{\phi_1},E(\bullet;\Psi,\lambda_1)>_\LMC
\overline{<\Ei_{\phi_1},E(\bullet;\Psi,\lambda_1)>_\LMC}\,\,d\lambda_1$$
where $\B_{\mathrm{cusp}}^K(P_1)$ is an orthonormal basis of the $K$-invariant cuspidal spectrum
for $P_1$ and $E(\bullet;\Psi,\lambda_1)$ is the Eisenstein series constructed from $\Psi$.
No new contribution comes from the other scalar products.
To show this one has to analyze the intertwining $\Mint(s,\lambda_1)$
operator attached to the unique element $s$ in $\weyl(\goth a_{P_1},\goth a_{P_2})$.
It would turn out that this operator is built out of $L$-functions for cuspidal representations
of $GL(2)$ that are holomorphic on the whole complex line and yield no further singularities
when moving the contour.

\bibliographystyle{amsalpha}

\begin{thebibliography}{99}

\bibitem{A1} 
  \textsc{Arthur, James} {\sl Eisenstein series and the trace formula}. Automorphic forms, 
  representations and L-functions (Proc. Sympos. Pure Math., Oregon State Univ., 
  Corvallis, Ore., 1977), Part 1, pp. 253-274, Proc. Sympos. Pure Math., XXXIII, 
  Amer. Math. Soc., Providence, R.I., 1979
  
\bibitem{A2} 
  \textsc{Arthur, James}
{\sl A trace formula for reductive groups. II. Applications of a truncation operator}.
Compositio Math. 40 (1980), no. 1, 87-121. 

\bibitem{B1} 
  \textsc{Borel, Armand}
{\sl Reduction Theory for Arithmetic Groups}. Algebraic Groups and Discontinuous Subgroups 
(Proc. Sympos. Pure Math., Boulder, Colo., 1965) pp. 20-25 Amer. Math. Soc., Providence, R.I. 1966 

\bibitem{B2} 
  \textsc{Borel, Armand}
{\sl Introduction to automorphic forms}.  Algebraic Groups and Discontinuous Subgroups 
(Proc. Sympos. Pure Math., Boulder, Colo., 1965) pp. 199-210 Amer. Math. Soc., Providence, R.I. 1966 

\bibitem{BJ} 
  \textsc{Borel, A.; Jacquet, H.}
{\sl Automorphic forms and automorphic representations.
With a supplement "On the notion of an automorphic representation'' by R. P. Langlands. }
Proc. Sympos. Pure Math., XXXIII, Automorphic forms, representations and L-functions 
(Proc. Sympos. Pure Math., Oregon State Univ., Corvallis, Ore., 1977), Part 1, pp. 189-207, 
Amer. Math. Soc., Providence, R.I., 1979. 

\bibitem{BL} 
  \textsc{Benstein, J.; Lapid, E.}
{\sl On the meromorphic continuation of Eisenstein series}
arXiv:1911.02342, 2019.

\bibitem{Fr} 
  \textsc{Franke, Jens}
{\sl Harmonic analysis in weighted L2-spaces. }
Ann. Sci. \'Ecole Norm. Sup. (4) 31 (1998), no. 2, 181-279. 

\bibitem{God}
 \textsc{Godement, Roger}
{\sl Analyse spectrale des fonctions modulaires.} [Spectral analysis of modular functions] 
SŽminaire Bourbaki, Vol. 9, Exp. No. 278, 15-40, Soc. Math. France, Paris, 1995. 

\bibitem{H}
 \textsc{Harish-Chandra}
{\sl Automorphic forms on Semisimple Lie Groups.} Lecture Notes in Mathematics, 
Vol. 62. Springer-Verlag, Berlin-New York, 1968


 \bibitem{K}  \textsc{Kottwitz, Robert E.}
{\sl Tamagawa numbers.}
Ann. of Math. (2) 127 (1988), no. 3, 629-646. 

 \bibitem{Ku}  \textsc{Kubota, Tomio}
{\sl Elementary theory of Eisenstein series.} Kodansha Ltd., Tokyo; Halsted Press [John Wiley \& Sons], 
New York-London-Sydney, 1973.  

 \bibitem{Lai}  \textsc{Lai, K. F.}
{\sl Tamagawa number of reductive algebraic groups.}
Compositio Math. 41 (1980), no. 2, 153-188

 \bibitem{LW}  \textsc{Labesse, Jean-Pierre; Waldspurger, Jean-Loup}
{\sl La formule des traces tordue d'aprs le Friday Morning Seminar. 
With a foreword by Robert Langlands. }
CRM Monograph Series, 31. American Mathematical Society, Providence, RI, 2013 


\bibitem{LES}
 \textsc{Langlands, Robert P.}
{\sl On the functional equations satisfied by Eisenstein series.} Lecture Notes in Mathematics, 
Vol. 544. Springer-Verlag, Berlin-New York, 1976


\bibitem{LB1} 
  \textsc{Langlands, Robert P.} {\sl The volume of the fundamental domain for some arithmetical 
  subgroups of Chevalley groups. } Algebraic Groups and Discontinuous Subgroups 
  (Proc. Sympos. Pure Math., Boulder, Colo., 1965) pp. 143-148 Amer. Math. Soc., Providence, R.I. 1966 

 \bibitem{LB2}
  \textsc{Langlands, Robert P.}{\sl  Eisenstein series. } Algebraic Groups and Discontinuous 
  Subgroups (Proc. Sympos. Pure Math., Boulder, Colo., 1965) pp. 235-252 
  Amer. Math. Soc., Providence, R.I. 1966


 \bibitem{LB3}
 \textsc{Langlands, Robert P.} {\sl Dimension of spaces of automorphic forms.} 
 Algebraic Groups and Discontinuous Subgroups (Proc. Sympos. Pure Math., 
 Boulder, Colo., 1965) pp. 253-257 Amer. Math. Soc., Providence, R.I. 1966 
 
 \bibitem{LEP}
 \textsc{Langlands, Robert P.} {\sl Euler products.} A James K. Whittemore 
 Lecture in Mathematics given at Yale University, 1967. 
 Yale Mathematical Monographs, 1. Yale University Press, New Haven, Conn.

\bibitem{LLMA}
 \textsc{Langlands, Robert P.} {\sl Problems in the theory of automorphic forms. }
 Lectures in modern analysis and applications, III, pp. 18-61. 
 Lecture Notes in Math., Vol. 170, Springer, Berlin, 1970. 
 
 \bibitem{MW1}
 \textsc{MÏglin, C.; Waldspurger, J.-L.}
{\sl  Le spectre rŽsiduel de GL(n). }
Ann. Sci. \'Ecole Norm. Sup. (4) 22 (1989), no. 4, 605-674. 

 \bibitem{MW2}
 \textsc{MÏglin, C.; Waldspurger, J.-L.}
{\sl Spectral decomposition and Eisenstein series.}
Cambridge Tracts in Mathematics, 113. Cambridge University Press, Cambridge, 1995. xxviii+338 pp.


  \bibitem{Rog}
 \textsc{Rogawski, Jonathan D.}
{\sl Automorphic representations of unitary groups in three variables.}
Annals of Mathematics Studies, 123. Princeton University Press, Princeton, NJ, 1990. 

\bibitem{Sc}
 \textsc{Schiffmann, GŽrard}
{\sl IntŽgrales d'entrelacement et fonctions de Whittaker. }
Bull. Soc. Math. France 99 (1971), 3-72. 

 \bibitem{S1}
 \textsc{Selberg, Atle}
{\sl Discontinuous groups and harmonic analysis.}
 Proc. Internat. Congr. Mathematicians (Stockholm, 1962) pp. 177-189. 1963 
 
 \bibitem{Tits}
 \textsc{Tits, J.}
{\sl Reductive groups over local fields.}
 Automorphic forms, representations and L-functions 
 (Proc. Sympos. Pure Math., Oregon State Univ., Corvallis, Ore., 1977), Part 1, pp. 29-69,
Proc. Sympos. Pure Math., XXXIII, Amer. Math. Soc., Providence, R.I., 1979.


\end{thebibliography}

\end{document}